\documentclass[review=false, screen=false, nonacm]{acmart}

\usepackage{booktabs} 

\usepackage[ruled]{algorithm2e} 

\usepackage[normalem]{ulem}
\usepackage{enumitem}
\usepackage{cancel}

\SetAlFnt{\small}
\SetAlCapFnt{\small}
\SetAlCapNameFnt{\small}
\SetAlCapHSkip{0pt}
\IncMargin{-\parindent}
\newcommand{\E}{\mathbb{E}}
\newcommand{\R}{\mathbb{R}}

\renewcommand{\P}{\mathbb{P}}
\newtheorem*{remark}{Remark}
\newtheorem{assumption}{Assumption}

\usepackage{subcaption}
\usepackage[normalem]{ulem}
\acmJournal{TWEB}
\acmVolume{?}
\acmNumber{?}
\acmArticle{?}
\acmYear{2018}
\acmMonth{5}
\copyrightyear{?}

\setcopyright{acmlicensed}

\allowdisplaybreaks

\begin{document}
\bibliographystyle{ACM-Reference-Format}

\author{Tim Hellemans}
\author{Benny Van Houdt}
\affiliation{%
  \institution{University Of Antwerp}
  \streetaddress{Middelheimlaan 1}
  \city{Antwerp}
  \postcode{2000}
  \country{Belgium}}
\title[Mean waiting time in large-scale and critically loaded power of d load balancing systems]
{Mean Waiting Time in Large-Scale and Critically Loaded Power of d Load Balancing Systems}

\begin{abstract}
Mean field models are a popular tool used to analyse load balancing policies. In some exceptional cases the waiting time distribution of the mean field limit has an explicit form. In other cases it can be computed as the solution of a set of differential equations. In this paper we study the limit of the mean waiting time $\E[W_\lambda]$ as the arrival rate $\lambda$ approaches $1$ 
for a number of load balancing policies when job sizes are exponential with
mean $1$ (i.e.~when the system gets close to instability). As $\E[W_\lambda]$ diverges to infinity, we scale with $-\log(1-\lambda)$ and present a method to compute the limit $\lim_{\lambda\rightarrow 1^-}-\E[W_\lambda]/\log(1-\lambda)$. We
show that this limit has a surprisingly simple form for the load balancing algorithms
considered. 

More specifically, we present a general result that holds for any policy for which the associated differential equation satisfies a list of assumptions. For the well-known LL($d$) policy which assigns an incoming job to a server with the least work left among $d$ randomly
selected servers these assumptions are trivially verified. For this policy we prove the limit is given by $\frac{1}{d-1}$. We further show that the LL($d,K$) policy,
which assigns batches of $K$ jobs to the $K$ least loaded servers among $d$ randomly selected servers,  
satisfies the assumptions and the limit is equal to $\frac{K}{d-K}$. For a policy which applies LL($d_i$) with probability $p_i$, we show that the limit is given by $\frac{1}{\sum_i p_i d_i - 1}$. We further indicate that our main result can also be used for load balancers with
redundancy or memory.

In addition, we propose an alternate scaling $-\log(p_\lambda)$
instead of $-\log(1-\lambda)$, where $p_\lambda$ is adapted to the policy at hand, 
such that $\lim_{\lambda\rightarrow 1^-}-\E[W_\lambda]/\log(1-\lambda)=\lim_{\lambda\rightarrow 1^-}-\E[W_\lambda]/\log(p_\lambda)$, where the limit $\lim_{\lambda\rightarrow 0^+}-\E[W_\lambda]/\log(p_\lambda)$ is well defined and non-zero (contrary to
$\lim_{\lambda\rightarrow 0^+}-\E[W_\lambda]/\log(1-\lambda)$). This allows to obtain relatively flat curves for $-\E[W_\lambda]/\log(p_\lambda)$ for $ \lambda \in [0,1]$ which indicates that the low and high load
limits can be used as an approximation when $\lambda$ is close to one or zero.

Our results rely on the earlier proven ansatz which asserts that 
for certain load balancing policies the workload distribution of any finite set of queues becomes independent of one another as the number of servers tends to infinity. 
\end{abstract}

\maketitle

\section{Introduction}
Load balancing plays an important role in large scale data networks, server farms, cloud and grid computing. From a mathematical point of view, load balancing policies can be split into two main categories. The first category exists of queue length dependent load balancing policies where the dispatcher collects some information on the number of jobs in some servers and assigns an incoming job using this information. A well studied example of this policy type is the SQ($d$) policy, where an incoming job is assigned to the shortest among $d$ randomly selected servers (see e.g.~\cite{mitzenmacher2001power, vvedenskaya3}). The second category, which is our main focus, consists of workload dependent load balancing policies, for these policies the dispatcher balances the load on the servers by employing information on the amount of work that is left on some of the servers (see also \cite{hellemans2019workload}). This can be done explicitly if we assume the amount of work on servers is known or implicitly by employing some form of redundancy such as e.g.~cancellation on start or late binding (see also \cite{ousterhout2013sparrow}). A well studied policy of this type is the LL($d$) policy, where each incoming job joins the server with the least amount of work left out of $d$ randomly sampled servers (see e.g.~\cite{hellemans2018power}).

In order to compute performance metrics such as the mean waiting time, the waiting time distribution, etc.~most work relies on mean-field models \cite{kurtz1, shneer2020large, hellemans2019workload, jinan2020load, bramsonLB}. Mean field models capture
the limiting stationary behavior of the system as the number of servers tends to infinity
provided that any finite set of servers becomes independent and identically distributed.  Recently this independence was proven for a wide variety of workload dependent load balancing policies in \cite{shneer2020large}. All but one of the workload dependent policies studied in this work fit into the framework of \cite{shneer2020large}. The limiting stationary
workload can therefore be described by the stationary workload distribution of a single server/queue. In order to analyse this queue, termed {\it the queue at the cavity}, the stationary workload distribution is characterized by an Integro Differential Equation, which can sometimes be simplified to a one dimensional Ordinary Differential Equation (ODE) in case job sizes are exponential. Throughout this paper, we assume the job size distribution is exponential with mean one.

We relate to each system size $N$ an arrival rate $\lambda_N$. To obtain the mean field limit as described earlier, one sets $\lambda_N=\lambda N$ for some fixed $\lambda < 1$. One is often interested in the behaviour of the queueing system as the system approaches its critical load. To study this, one could set $\lambda_N = \lambda(N) N$ where $\lambda(N) \rightarrow 1^-$ as $N$ tends to infinity. This approach was for example used in \cite{liu2020steady1, liu2020steady2, brightwell2012supermarket, eschenfeldt2018join} to study the SQ($d$) model in heavy traffic. Another approach, which is the one we use here, is to first obtain the stationary distribution of the mean field model with a fixed $\lambda(N)=\lambda < 1$ and subsequently take the limit $\lambda \rightarrow 1^-$ of the resulting mean field models. For workload dependent policies, we are not aware of any work where the approach of letting $\lambda(N) \rightarrow 1^-$ has been considered. For the SQ($d$) policy, it is shown in \cite{mitzenmacher2001power} that $\lim_{\lambda\rightarrow 1^-} - \frac{\E[W_\lambda]}{\log(1-\lambda)} = \frac{1}{\log(d)}$, with $W_\lambda$ the waiting time distribution for the SQ($d$) policy with arrival rate $\lambda$. However, its proof is a technical computation which relies heavily on the closed form solution of the stationary distribution and does not seem to generalize well.

In this paper we establish a general result which can be employed to obtain the limit: \begin{equation} \label{eq:gen_result_intro}
\lim_{\lambda \rightarrow 1^-} -\frac{\E[W_\lambda]}{\log(1-\lambda)},
\end{equation}
where $W_\lambda$ is the waiting time distribution of a workload dependent load balancing policy (see Theorem \ref{thm:gen_result_ODE} and Corollary \ref{cor:gen_result}). This value can be used as a reference to indicate how well a policy behaves under a high load.
As we divide by $-\log(1-\lambda)$, we are focussing on load balancing policies where
an exponential improvement in the mean waiting time is expected compared to random assignment.
For LL($d$) it is indirectly claimed in \cite{hellemans2018power} that the limit \eqref{eq:gen_result_intro} is given by $\frac{1}{d-1}$, though the proof is incorrect (c.f.~the remark after Corollary \ref{cor:LLd}). 
Our result  provides  a list of sufficient assumptions under which the
limit in \eqref{eq:gen_result_intro} can be computed in a straightforward manner. Although computing the limit is easy, verifying the listed assumptions may present quite a challenge, one of our main contributions is establishing these assumptions for LL($d,K$).

We start by applying our method on LL($d$) providing a first proof for the associated limit. 
We then apply our method to the LL($d,K$) policy (see also \cite{van2019global, ying2017power}). For this policy, jobs are assumed to arrive in batches of size $K$, we then sample $d > K$ servers and the jobs are assigned to the $K$ queues with the least amount of work left. We show in Section \ref{sec:LLdK} that:
\begin{align}
 \lim_{\lambda \rightarrow 1^-} -\frac{\E[W_\lambda]}{\log(1-\lambda)} = \frac{1}{\frac{d}{K} - 1}=\frac{K}{d-K}. \label{eq:lim_LLdK}
\end{align}
for LL($d,K$).
One of the main technical contributions of the paper, apart from establishing Theorem
\ref{thm:gen_result_ODE}, exists in verifying the third assumption of this theorem for LL($d,K$).

Next, we consider the LL($d_1,\dots,d_n,p_1,\dots,p_n$) policy, where with probability $p_i$ we select $d_i$ servers and assign the incoming job to the queue with the least amount of work amongst these $d_i$ selected servers.
We show that for LL($d_1,\dots,d_n,p_1,\dots,p_n$) we have
\begin{align*}
 \lim_{\lambda \rightarrow 1^-} -\frac{\E[W_\lambda]}{\log(1-\lambda)} =\frac{1}{\sum_{i=1}^n p_i d_i - 1}.
\end{align*}
We observe that, when the system is highly loaded, the choice of $p_i$ and $d_i$ does not matter as long as the total amount of redundancy $\sum_{i=1}^n p_i d_i$ remains constant. Furthermore we find a general method to investigate which choice of $p_i$ and $d_i$ yields smaller response times when $\lambda < 1$.

In the special case of LL($1,d,1-p,p$), this policy applies the power of $d$ choices only to a proportion of the incoming jobs and assigns the other jobs arbitrarily. For this policy, we find that whenever $\lambda < 1$ the limiting probability that an arbitrary queue has workload at least $w$ is given by:
\begin{equation}\label{eq:exp_LLdp_WLdist}
\bar F(w)
=
\lambda \bigg[ \frac{1-(1-p)\lambda}{p \lambda^d +(1-(1-p)\lambda-p\lambda^d) e^{(d-1)(1-(1-p)\lambda)w}} \bigg]^{\frac{1}{d-1}},
\end{equation}
but no such solution appears to exist in general. This closed form expression yields an alternative method to obtain the limiting result. Equivalently this model may be described as having an individual arrival process with rate $(1-p)\lambda$ at each server in addition to an LL($d$) arrival stream with rate $p\lambda N$. This type of model was for example studied in \cite{bu2020approximations}.

We also argue that our result can be directly used
for Red($d$) with i.i.d.~replica's and we indicate how our result can be adapted for the SQ-variants of the policies we considered. Furthermore, we already used our general result to compute the same limit for load balancing policies with memory at the dispatcher (see \cite{hellemans2020memory}).

To obtain these results, the main insight we use is the fact that, as $\lambda$ approaches one, all queues have more or less the same amount of work (see also Figure \ref{fig0}). We are able to analytically approximate this amount of work, it represents how well a policy is able to balance loads under a high arrival rate. A similar observation was made in \cite{anton2019stability}, where it was noted that for Redundancy $d$ under Processor Sharing with identical replica's, the workload at all servers diverges to infinity at an equal rate when $\lambda$ exceeds $\frac{1}{d}$.

While $\lim_{\lambda\rightarrow 1^-}-\E[W_\lambda]/\log(1-\lambda)$ is finite and
non-zero, the scaling with $-\log(1-\lambda)$ is not very insightful when $\lambda$
is small as $\lim_{\lambda\rightarrow 0^+}-\E[W_\lambda]/\log(1-\lambda)$ tends to be zero.
We therefore additionally introduce an alternate scaling $-\log(p_\lambda)$, where
the value of $p_\lambda$ is policy specific and discussed in Section \ref{sec:low_traffic},
such that the limit for $\lambda$ tending to one remains the same, while 
$\lim_{\lambda\rightarrow 0^+}-\E[W_\lambda]/\log(p_\lambda)$ is well defined
and non-zero. It turns out that for this scaling the curve $-\E[W_\lambda]/\log(p_\lambda)$
is fairly flat for $\lambda \in [0,1]$ meaning that the low and high load
limits of the alternate scaling can be regarded as a good approximation for low and high loads.

The paper is structured as follows. In Section \ref{sec:ODE} we illustrate the type of policies for which our result is applicable. In Section \ref{sec:result}, we present the main result and indicate that it is applicable to the LL($d$) and Red($d$) policies. In Section \ref{sec:computation} we compute the limiting value as $\lambda \rightarrow 1^-$ for all considered policies and in Section \ref{sec:proof} we give the proof of the main result. In Section \ref{sec:LLdK} we verify the assumptions for LL($d,K$). In Section \ref{sec:l1ld} we cover LL($d_1,\dots,d_n,p_1,\dots,p_n$), here we also consider the case where $\lambda$ is bounded away from $1$ and the special case of LL($1, d, 1-p, p$). We introduce and discuss the alternate scaling by $-\log(p_\lambda)$ in Section \ref{sec:low_traffic}. We provide a selection of numerical experiments in Section \ref{sec:num_exp}. Conclusions are drawn and extensions are suggested in Section \ref{sec:concl}.

\section{General Result}\label{sec:general}
\subsection{The Ordinary Differential Equation} \label{sec:ODE}
Our main result can be applied to functions which are the solution of the ODE:
\begin{equation}
\bar F'(w) = T_\lambda(\bar F(w)) - \bar F(w), \label{eq:ODE}
\end{equation}
with some boundary condition $\bar F(0) =x$ with $x \in [\lambda,1]$. It was shown in \cite{hellemans2018power} that the ccdf of the limiting stationary workload distribution for the LL($d$) policy can be found as the solution of \eqref{eq:ODE} with boundary condition $\bar F(0) = \lambda$ and $T_\lambda(u) = \lambda u^d$. For the Red($d$) policy with i.i.d.~replicas it was proven in \cite{gardnerSIGM} that the ccdf of the response time distribution $\bar F_R(w)$ satisfies the same ODE, that is $\bar F_R(w)$ satisfies \eqref{eq:ODE} with $T_\lambda(u) = \lambda u^d$, but with boundary condition $\bar F_R(0) = 1$.

For LL($d, K$) (c.f.~Section \ref{sec:LLdK}) we find that the ccdf of the workload distribution satisfies \eqref{eq:ODE} with: 
\begin{equation} \label{eq:LLdK_Tlam}
T_\lambda(u) = \frac{\lambda}{K} \sum_{j=0}^{K-1} (K-j) \binom{d}{j} u^{d-j} (1-u)^j,
\end{equation}
and $\bar F(0)=\lambda$.
For the LL($d_1,\dots,d_n, p_1, \dots, p_n$) policy (c.f.~Section \ref{sec:LLd1..dn}) we find that the ccdf of the workload distribution is given by the solution of \eqref{eq:ODE} with:
\begin{equation}\label{eq:Fbar_LLd1dn}
T_\lambda(u) = \sum_i p_i u^{d_i},
\end{equation}
and $\bar F(0)=\lambda$.
For the memory dependent LL($d$) policy, we assume that each arrival has a probability $\pi_0(\lambda)$ to be routed using the LL($d$) policy, while with the remaining probability $1-\pi_0(\lambda)$ it is routed to an empty queue. For this policy we showed in \cite{hellemans2020memory} that the ccdf of the workload distribution $\bar F(w)$ satisfies \eqref{eq:ODE} with:
\begin{equation}\label{eq:Fbar_mem}
T_\lambda(u) = \lambda \pi_0(\lambda) u^d.
\end{equation}

The fact that the solution to \eqref{eq:ODE} captures the limit of the stationary 
workload distribution of a single queue, or the stationary response time distribution
of a job, as the number of queues $N$ tends to
infinity for the policies under consideration, is due to the recent results found in \cite{shneer2020large} (except for the memory dependent case, \eqref{eq:Fbar_mem}). In \cite{shneer2020large}, the authors prove the independence ansatz introduced in \cite{bramsonAAP} for a variety of workload dependent load balancing policies, their approach is based on the following three properties:
\begin{enumerate}[label=(\alph*)]
\item Monotonicity, which essentially states that as we increase the number of probes used per arrival, the delay a job experiences also reduces. \label{item:monotonicity}
\item Work conservation, that is, executed work is never lost.
\item The property that, on average, \textit{new arriving workload prefers to go to servers with lower workloads.} \label{item:lower_loads}
\end{enumerate}
More specifically, in \cite{shneer2020large} the authors prove the 
independence ansatz for any convex combination of LL($d, K$) (with arbitrary job sizes) and Red($d, K$) (with exponential job sizes and i.i.d.~replicas). For the memory dependent load balancing policies proving the ansatz remains an open problem as one is faced with the additional problem of a time-scale separation as the memory content evolves on a different time-scale than the workload. For more details we refer the reader to \cite{hellemans2020memory} and \cite{benaim2008class}.

\subsection{Statement $\&$ Application of the Main Result}\label{sec:result}
We first introduce all assumptions which we require in order to obtain the limiting value of $-\E[W_\lambda] / \log(1-\lambda)$ as $\lambda \rightarrow 1^-$. We illustrate the assumptions by showing that they hold for the LL($d$) and Red($d$) policies, that is for  the choice $T_\lambda(u) = \lambda u^d$.

\begin{assumption}\label{item:existulam}
There exists a $\bar \lambda \in (0,1)$ such that:
\begin{itemize}
\item For $\lambda\in (\bar \lambda, 1)$ there exists a $u \in (1,\infty): T_\lambda(u) = u$. We define $u_\lambda \in (1,\infty)$ as the minimal value for which $T_\lambda(u_\lambda) = u_\lambda$.
\item The function $u_\cdot : \lambda \rightarrow u_\lambda$ is continuous and $\lim_{\lambda \rightarrow 1^-} u_\lambda = 1$.
\end{itemize}
\end{assumption}
For $T_\lambda(u)=\lambda u^d$, we can set $\bar \lambda = 0$ and
finding $u_\lambda$ reduces to obtaining the smallest solution 
of $u = \lambda u^d$ in $(1,\infty)$. We quickly find that $
u_\lambda = \lambda^{\frac{1}{1-d}}$, which is obviously 
continuous in $\lambda$  and converges to one as $\lambda$ approaches $1$.
\begin{assumption}\label{item:Tu_smaller_u}
For all $u \in (0,1]$, we have:
\begin{itemize}
\item $T_\lambda(0) = 0$, $T_\lambda(u) < u$ and $\lim_{\lambda\rightarrow 1^-} \frac{T_\lambda(u)}{u} < 1$,
\item $\left( \frac{T_\lambda(u)}{u} \right) ' \geq 0$, which implies that $T_\lambda$ is increasing on $(0,1)$.
\end{itemize}
\end{assumption}
We have $\frac{T_\lambda(u)}{u} = \lambda u^{d-1}$ from which this assumption trivially follows.

\begin{assumption} \label{item:h_decreasing}
For all $\lambda \in ( \bar \lambda, 1)$ we define:
\begin{equation}\label{eq:def_hlam}
h_\lambda(x) = \frac{u_\lambda - T_\lambda(u_\lambda - x)}{x}.
\end{equation}
There is some $b \in \mathbb{N}$ such that for all $\lambda \in (\bar \lambda,1)$ we have $h_\lambda(x)$ is decreasing for $x \in [u_\lambda - \lambda^b,1)$.
\end{assumption}

For $T_\lambda(u) = \lambda u^d$, we find (with $h_\lambda(x)$ defined as in \eqref{eq:def_hlam}):
\begin{equation}\label{eq:hlam_LLd}
h_{\lambda}(x)
=
\frac{\lambda^{\frac{1}{1-d}} - \lambda (\lambda^{\frac{1}{1-d}} - x)^d }{x},
\end{equation}
its derivative is given by:
$$
h_\lambda'(x)
=
\frac{\lambda \left(\lambda^{\frac{1}{1-d}}-x\right)^{d-1} \left(\lambda^{\frac{1}{1-d}}+(d-1)
   x\right)-\lambda^{\frac{1}{1-d}}}{x^2}
$$
differentiating $x^2 h_\lambda'(x)$ once more yields:
$$
(x^2 h_\lambda'(x))'=-\lambda(d-1)d\left(\lambda^{\frac{1}{1-d}} -x \right)^{d-2} x,
$$
which is obviously negative for $x \in [0,1)$. Hence, this assumption now follows with $b=0$ from the fact that $(x^2 h_\lambda'(x))$ equals $0$ for $x=0$.

\begin{assumption}\label{item:ODE_only_new_req}
For any $\lambda \in (\bar \lambda, 1)$ we let $\bar w_\lambda \in [0,\infty)$ be the smallest value for which $\bar F(\bar w_\lambda) \leq \lambda^b$. There is some $\bar w$ which can be chosen independently of $\lambda$ such that $\bar w_\lambda \leq \bar w$.
\end{assumption}
As we showed assumption \ref{item:h_decreasing} with $b=0$, we find that $\bar w_\lambda=0$ for all $\lambda \in [0,1)$ from which this assumption trivially follows with $\bar w = 0$.
\begin{remark}
For assumption \ref{item:ODE_only_new_req} it suffices in general to show that $\bar F(w) \leq \lambda e^{-(1-\lambda) w }$. Indeed, to have $\lambda e^{-(1-\lambda) w} \leq \lambda^b$ it suffices to have $b-1  \leq \bar w_\lambda$. Therefore one may pick $\bar w=b-1$. Note that $\lambda e^{-(1-\lambda)w}$ is the probability that the workload of an M/M/1 queue is at least $w$, therefore it suffices that the policy is at least as good as random routing.
\end{remark}

\begin{assumption} \label{item:lim_hlam}
There is some $A \in (1,\infty)$ for which $\lim_{\lambda\rightarrow 1^-} h_\lambda(u_\lambda-\lambda^b)=A$.
\end{assumption}
For $h_\lambda(x)$ given by \eqref{eq:hlam_LLd} we note that:
$$
h_\lambda(u_\lambda - 1)=\frac{\lambda^{\frac{1}{1-d}}-\lambda}{\lambda^{\frac{1}{1-d}}-1} \underset{\lambda \rightarrow 1^-}{\longrightarrow} d,
$$
where the limit statement can be shown using l'Hopital's rule. Therefore this assumption holds for LL($d$) and Red($d$) with $A=d$.

\begin{assumption}\label{item:lim_ulam}
There is some $B \in [0,\infty)$ for which $\lim_{\lambda\rightarrow 1^-} \frac{\log(u_\lambda-\lambda^b)}{\log(1-\lambda)}=B$.
\end{assumption}
Using $u_\lambda = \lambda^{\frac{1}{1-d}}$ and $b=0$, we find that $B=1$, when
$T_\lambda(u)= \lambda u^d$, by a simple application of l'Hopital's rule.

\begin{assumption} \label{item:lim_heps}
We have $\lim_{\varepsilon \rightarrow 0^+} \lim_{\lambda\rightarrow 1^-} h_{\lambda}(\varepsilon) = A$.
\end{assumption}
We note that for $T_\lambda(u)=\lambda u^d$:
$$
\lim_{\lambda \rightarrow 1^-} h_{\lambda}(\varepsilon)
=
\frac{1-(1-\varepsilon)^d}{\varepsilon}
\underset{\varepsilon \rightarrow 0^+}{\longrightarrow} d,
$$
from which assumption \ref{item:lim_heps} follows. We are now in a position to state our general result.
\begin{theorem}\label{thm:gen_result_ODE}
For any $\lambda \in (0,1)$ we let $\bar F:[0,\infty) \rightarrow [0,1]$ be a solution to \eqref{eq:ODE} with $\bar F(0)=x$ for some fixed $x \in [\lambda, 1]$, where we assume $\bar F$ is the unique continuously differentiable solution to this ODE. Further, we assume that $T_\lambda$ satisfies Assumptions \ref{item:existulam} - \ref{item:lim_heps}. We then have:
\begin{equation}\label{eq:gen_lim}
\lim_{\lambda\rightarrow 1^-} - \frac{\int_0^\infty \bar F(w)\, dw}{\log(1-\lambda)} = \frac{B}{A-1}.
\end{equation}
\end{theorem}

For most of our applications, $\int_0^\infty \bar F(w) \, dw$ is equal to the expected workload.
 The next Corollary shows  that the mean queue length is in fact equal to the mean workload
 for some of the policies considered in this paper, allowing us to obtain the mean waiting
 time from the mean workload using Little's law.
 
\begin{corollary}\label{cor:gen_result}
Consider LL($d$), LL($d,K$) or $LL(d_1,\ldots,d_n,p_1,\ldots,p_n)$ with
job sizes that are exponential with mean one and assume 
the ccdf of the workload distribution $\bar F(w)$ satisfies the requirements outlined in Theorem \ref{thm:gen_result_ODE}, then the mean queue length is equal to the mean workload. In particular:
\begin{equation}\label{eq:gen_limit_response}
\lim_{\lambda\rightarrow 1^-} - \frac{\E[W_\lambda]}{\log(1-\lambda)} 
= \lim_{\lambda\rightarrow 1^-} - \frac{\E[R_\lambda]}{\log(1-\lambda)} 
= \lim_{\lambda\rightarrow 1^-} - \frac{\E[Q_\lambda]}{\log(1-\lambda)}
= \lim_{\lambda\rightarrow 1^-} - \frac{\E[L_\lambda]}{\log(1-\lambda)}
= \frac{B}{A-1},
\end{equation}
where $W_\lambda$, $R_\lambda$, $Q_\lambda$ and $L_\lambda$ denote the waiting time, response time, queue length and workload distribution for the load balancing policy with load $\lambda$.
\end{corollary}
\begin{proof}
We first note that if $\bar F'(w) = T_\lambda(\bar{F}(w)) - \bar F(w)$
and $\bar F(0) = \lambda$, then $\bar F(w)$ also satisfies the following fixed point equation:
$$
\bar F(w) = \lambda - \lambda \int_0^w \left( 1 - \frac{T_\lambda(\bar F(u))}{\lambda} \right) e^{u-w} \, du,
$$
which can be seen by replacing $T_\lambda(\bar F(u))$ by $\bar F'(u)+F(u)$ and using integration
by parts. This fixed point equation can be further simplified to:
$$
\bar F(w) = \lambda e^{-w} + \int_0^w T_\lambda(\bar F(u)) e^{u-w} \, du.
$$
Integrating both sides from $0$ to infinity, we obtain (using Fubini):
\begin{equation} \label{eq:proof_resp1}
\frac{1}{\lambda} \int_0^\infty \bar F(w) \, dw = 1 + \int_0^\infty \frac{T_\lambda(\bar F(w))}{\lambda} \, dw.
\end{equation}
One can see that for the policies considered $T_\lambda(\bar F(w))$ is the arrival rate to servers with $w$ or more work, from this it follows that $\frac{T_\lambda(\bar F(w))}{\lambda}$ is the probability an arbitrary arrival has a waiting time which exceeds $w$. We can thus write \eqref{eq:proof_resp1} as $\frac{\E[L_\lambda]}{\lambda} = 1 + \E[W_\lambda] = \E[R_\lambda]$.

From this observation, combining \eqref{eq:proof_resp1} and Little's law, it follows that the mean workload is indeed equal to the mean queue length. The equations given in \eqref{eq:gen_limit_response} now easily follow, indeed the first equality follows from $\E[R_\lambda] = \E[W_\lambda] + 1$, the second equality is Little's Law, the third equality is what we just proved and the last equality follows by applying Theorem \ref{thm:gen_result_ODE}.
\end{proof}

As we already showed all assumptions for LL($d$), it follows by applying Corollary \ref{cor:gen_result} that:
\begin{corollary}\label{cor:LLd}
Let $W_\lambda$ denote the waiting time distribution for the LL($d$) policy with arrival rate $\lambda$ and exponential job sizes with mean one, we find:
\begin{align}\label{eq:lim_LLd}
\lim_{\lambda\rightarrow 1^-} - \frac{\E[W_\lambda]}{\log(1-\lambda)} = \frac{1}{d-1}
\end{align}
\end{corollary}
\begin{remark}
The equality in \eqref{eq:lim_LLd} appears in Theorem 7.2 found in \cite{hellemans2018power}, however there is an incorrect use of the Moore-Osgood Theorem, as the limit function $U$ is not necessarily continuous. In fact, its continuity is exactly what needs to be shown. Therefore this paper presented the \textit{first complete proof} of this result.
\end{remark}

For the Red($d$) policy with i.i.d.~replicas the waiting time is not clearly defined, therefore we state this result w.r.t.~response time, we find from Theorem \ref{thm:gen_result_ODE}:
\begin{corollary}
Let $R_\lambda$ denote the response time distribution for the Red($d$) policy with i.i.d.~replicas, arrival rate $\lambda$ and exponential job sizes with mean one, we find:
$$
\lim_{\lambda \rightarrow 1^-} - \frac{\E[R_\lambda]}{\log(1-\lambda)} = \frac{1}{d-1}
$$
\end{corollary}

For the memory scheme, in the particular case where servers probe the dispatcher when they become idle, we showed in \cite{hellemans2020memory} that $\pi_0(\lambda) = \frac{1- (1-\lambda^d)^{\frac{1}{M+1}}}{\lambda^d}$ (with $M$ the memory size). From this we easily computed the values $A = d$ and $B = \frac{1}{M+1}$ and it follows that:
$$
\lim_{\lambda \rightarrow 1^-} - \frac{\E[L_\lambda]}{\log(1-\lambda)} = \frac{1}{M+1} \cdot \frac{1}{d-1}.
$$

In the next section we show that computing the values of $A$ and $B$ is in general
not hard. Thus our result allows one to quickly obtain an expression for the
limiting value $B/(A-1)$. Of course to formally prove that it is the correct limiting value,
the assumptions must be verified.

\subsection{Computation of the limit}\label{sec:computation}
\subsubsection{LL($d, K$)}
For this policy $T_\lambda(u)$ is given by equation \eqref{eq:LLdK_Tlam}. First note that by differentiating $u_\lambda = T_\lambda(u_\lambda)$ one obtains $\lim_{\lambda\rightarrow 1^-} u_\lambda' =-K/(d-K)$ (see also \eqref{eq:Dulam_LLdk_paper}).

We now compute the value $A$ of assumption \ref{item:lim_hlam}, to this end we note that:
$$
\lim_{\lambda \rightarrow 1^-} h_{\lambda} (u_\lambda - \lambda^b)=\lim_{\lambda \rightarrow 1^-} \frac{T_{\lambda}(u_\lambda) - T_\lambda(\lambda^b)}{u_\lambda-\lambda^b}.
$$
Furthermore, one can see that in both $\frac{T_{\lambda}(u_\lambda)} {u_\lambda-\lambda^b}$ and  $\frac{T_{\lambda}(\lambda^b)} {u_\lambda-\lambda^b}$ the  terms with $j\geq 2$ 
of $T_\lambda$ disappear in the limit of $\lambda \rightarrow 1^-$. Therefore we find that:
\begin{align*}
\lim_{\lambda\rightarrow 1^-} h_{\lambda}(u_\lambda - \lambda^b)
&=
\lim_{\lambda \rightarrow 1^-} \frac{\lambda}{K} \bigg[
K \frac{u_{\lambda}^d - \lambda^{bd}}{u_\lambda - \lambda^b} + (K-1) d \frac{u_{\lambda}^{d-1} (1-u_{\lambda}) - \lambda^{b(d-1)} (1-\lambda^b)}{u_\lambda - \lambda^b} \bigg].
\end{align*}
The result now follows by applying l'Hopital's rule to conclude that:
$$
\lim_{\lambda \rightarrow 1^-} h_{\lambda}(u_\lambda - \lambda^b)
=
\frac{1}{K} \left( Kd + (K-1) d (-1) \right) = \frac{d}{K},
$$
completing the proof of assumption \ref{item:lim_hlam} with $A = \frac{d}{K}$.

For assumption \ref{item:lim_ulam} it follows by applying l'Hopital's rule that $B=1$. From this one may conjecture that the limiting value:
\begin{equation} \label{eq:lim_LLdK}
\lim_{\lambda \rightarrow 1^-} - \frac{\E[W_\lambda]}{\log(1-\lambda)} = \frac{1}{\frac{d}{K}-1} = \frac{K}{d-K}
\end{equation}
holds. In Section \ref{sec:LLdK} we verify the other assumptions (proving the above limit). Assumption \ref{item:h_decreasing} is particularly difficult to show for LL($d, K$).

\subsubsection{LL($d_1,\dots, d_n, p_1,\dots, p_n$)}
For this policy we have $T_\lambda(u)$ given by \eqref{eq:Fbar_LLd1dn}, from this we find that
\begin{align*}
u_{\lambda}'
&= \frac{u_{\lambda}}{\lambda} + \lambda \sum_{i=1}^n p_i d_i u_{\lambda}^{d_i-1} u_{\lambda}'.
\end{align*}
Taking the limit $\lambda \rightarrow 1^-$ one finds that $\lim_{\lambda\rightarrow 1^-} u_\lambda' = \frac{-1}{\sum_i p_i d_i -1}$. From this one can show that assumption \ref{item:lim_hlam} holds with $A = \sum_i p_i d_i$ and assumption \ref{item:lim_ulam} with $B=1$. We verify the remaining assumptions in Section \ref{sec:LLd1..dn}.

\begin{remark}
For the queue length dependent variants one denotes by $u_k$ the probability that the cavity queue has $k$ or more jobs in its server. These $u_k$ can be found from the recursive relation $u_{k+1} = T_\lambda(u_k)$, where the function $T_\lambda$ is given by the same function as the one for the workload dependent variant. For example, for SQ($d, K$) one can show that $u_{k+1} = T_\lambda(u_k)$, where $T_\lambda$ is defined as in \eqref{eq:LLdK_Tlam}. This allows one to construct a proof, based on the same idea as the one illustrated in Figure \ref{fig0}, that for the queue length dependent variants one has:
$$
\lim_{\lambda \rightarrow 1^-} - \frac{\E[W_\lambda]}{\log(1-\lambda)} = \frac{B}{\log(A)}.
$$
In particular, for the SQ($d, K$) policy this shows that $\lim_{\lambda \rightarrow 1^-} - \frac{\E[W_\lambda]}{\log(1-\lambda)} = \frac{1}{\log(\frac{d}{K})}$. Noting that $A-1 < \log(A)$ for all $A > 1$, this shows that in the limit $\lambda \rightarrow 1^-$ the workload based variants always yield smaller waiting times.
\end{remark}

\subsection{Proof of Theorem \ref{thm:gen_result_ODE}} \label{sec:proof}
\begin{figure*}[t]
\begin{subfigure}{0.45\textwidth}
\centering
\captionsetup{width=.8\linewidth}
\includegraphics[width=1\linewidth]{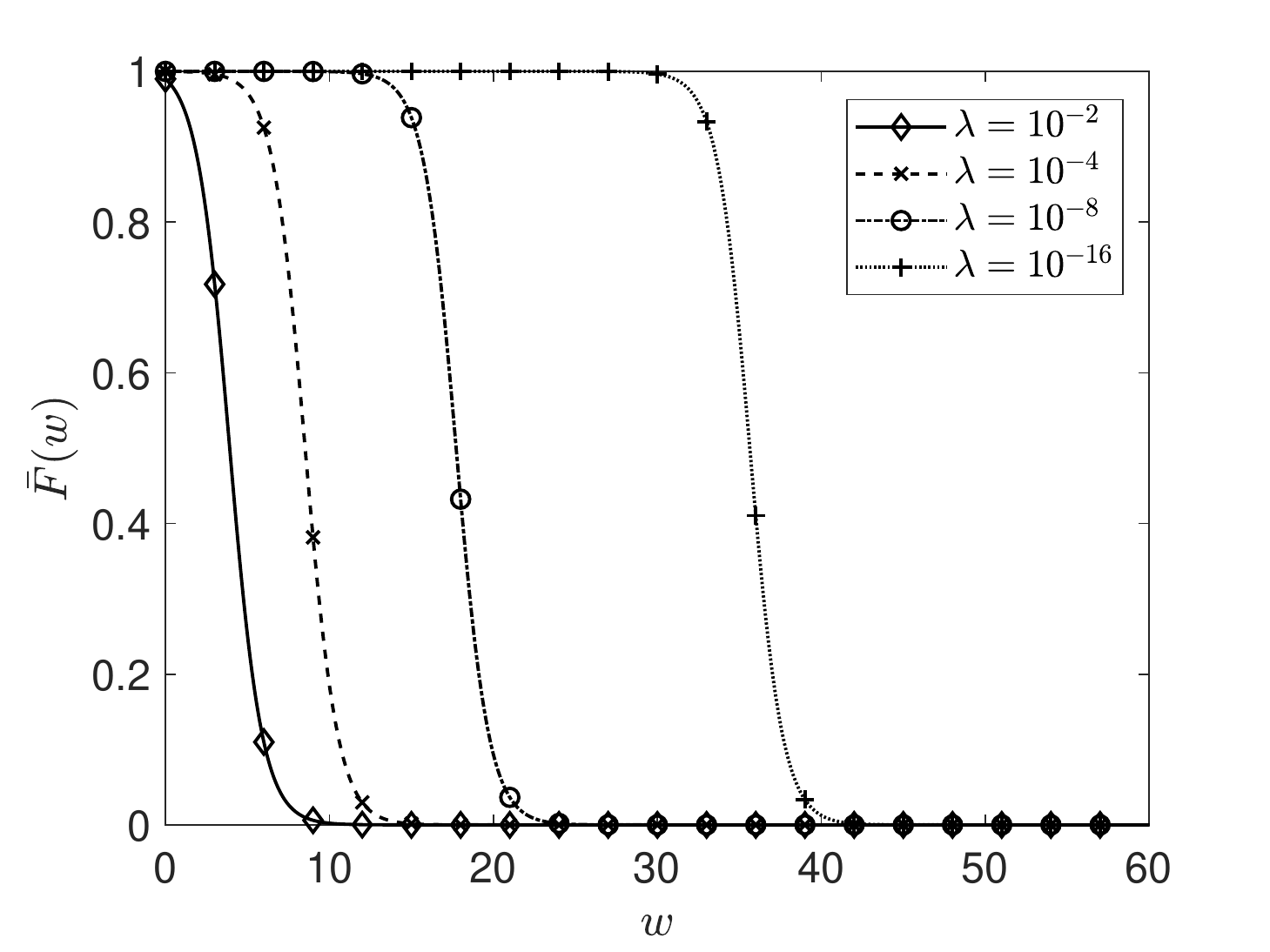}
\caption{ Plot of $\bar F(w)$ for $d=2$ and various values of $\lambda \approx 1$.}
\label{fig0a}
\end{subfigure}
\begin{subfigure}{0.45\textwidth}
\centering
\captionsetup{width=.8\linewidth}
\includegraphics[width=1\linewidth]{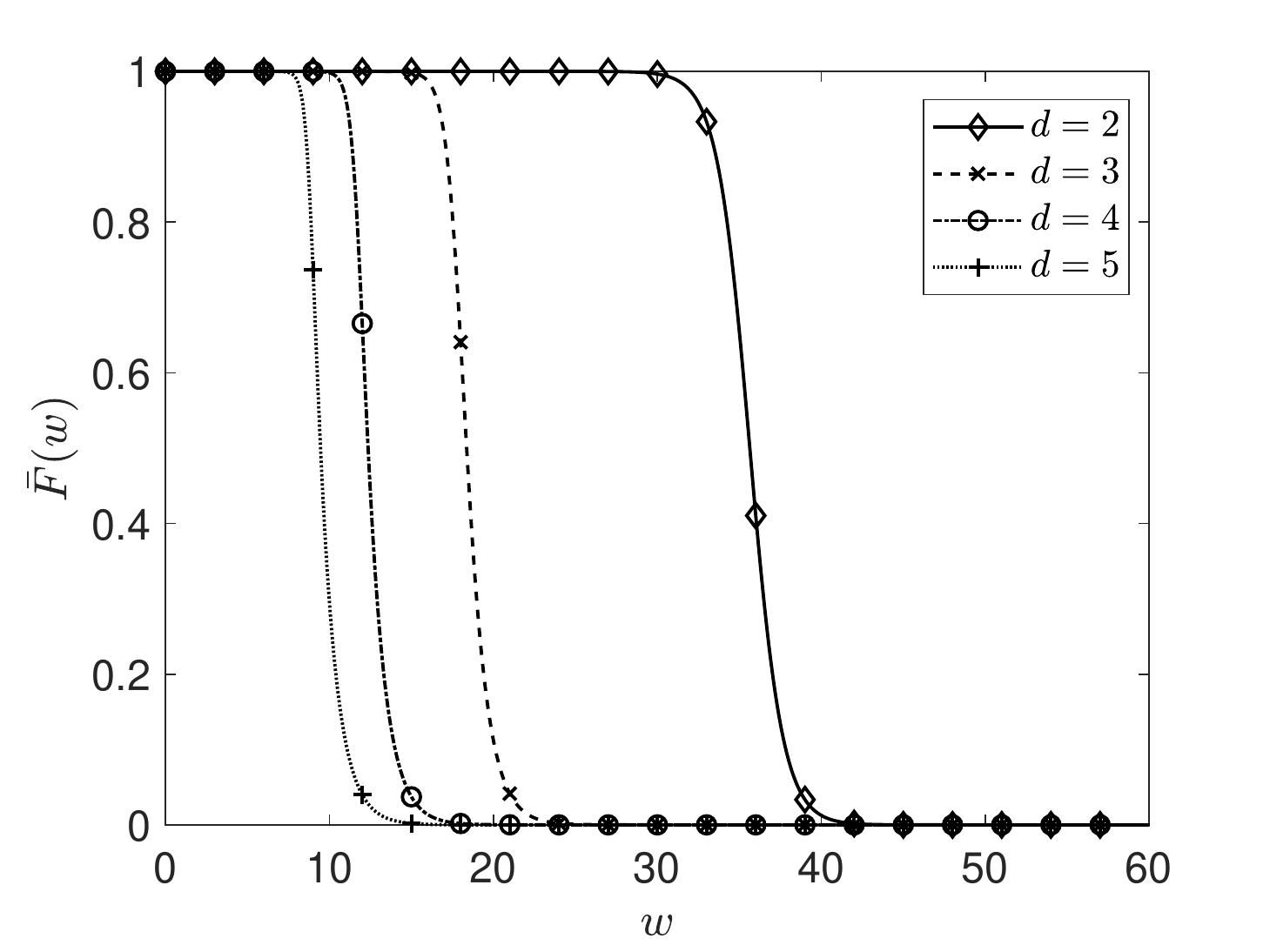}
\caption{ Plot of $\bar F(w)$ for $\lambda = 1-10^{-16}$ and various values of $d$.}
\label{fig0b}
\end{subfigure}
\caption{This figure illustrates that the tail behaviour of $\bar F(w)$ is the same for all values of $\lambda$, this is the main idea used in Theorem \ref{thm:gen_result_ODE}.}
\label{fig0}
\end{figure*}
The main idea used in the proof of Theorem \ref{thm:gen_result_ODE} is the fact that the tail behaviour of $\bar F(w)$ is identical for all values of $\lambda$, while the point at which the tail \textit{initiates its descend} moves further to the right as the value of $\lambda$ approaches $1$. This can be seen in Figure \ref{fig0a}, where we plot $\bar F(w)$ for $d=2$ and $\lambda=1-10^{-2}, 1-10^{-4}, 1-10^{-8}$ and $1-10^{-16}$. Moreover, we observe in Figure \ref{fig0b} that increasing the value of $d$, moves the tail of the function $\bar F(w)$ to the left which corresponds to having a smaller expected waiting time when $\lambda \approx 1$. Our proof boils down to formalizing the idea that there exists some $w_{\lambda}$ such that $\bar F(w) \approx 1$ for $w \leq w_\lambda$, while the integral $\int_{w_{\lambda}}^\infty \bar F(w) \, dw$ remains bounded for all $\lambda$.  

\begin{proof}
Our strategy exists in showing that $\bar F(w)$ stays close to one for a long enough time and then decays sufficiently fast to zero. Throughout the proof, we assume that $\lambda\in (\bar \lambda, 1)$. 
Due to assumption \ref{item:Tu_smaller_u} and $\bar F(w) \in [0,1]$, 
we find that $\bar F(w)$ is decreasing on $[0,\infty)$
and therefore $\lim_{w \rightarrow \infty} \bar F(w) = z$ exists.
As $0 = \lim_{w \rightarrow \infty} \bar F'(w)= \lim_{w \rightarrow \infty} T_\lambda(
\bar F(w))-z$ and $T_\lambda$ is continuous,  we have $0= T_\lambda(z)-z$. Hence, assumption \ref{item:Tu_smaller_u} yields that $z=0$.

Define $u_\lambda$ as in assumption \ref{item:existulam} and let $H(w)= u_\lambda - \bar F(w)$. We find:
$$
H'(w) = u_\lambda - T_\lambda(u_\lambda - H(w)) - H(w),
$$
therefore we have $\frac{H'(w)}{H(w)} = h_{\lambda}(H(w))-1$. 
For any $w \geq \bar w_{\lambda}$, we have $H(w) \geq u_\lambda - \lambda^b$
and due to assumption \ref{item:h_decreasing} this yields :
\begin{equation}
\frac{H'(w)}{H(w)} = h_{\lambda}(H(w))-1 \leq h_{\lambda}(u_\lambda - \lambda^b) - 1. \label{eq:H_prime_over_H_upper}
\end{equation}
Now let $0 < \varepsilon < 1$ be arbitrary. As $H(w)$ increases (from $u_\lambda-\bar F(0)$ to 
$u_\lambda$), we can define $w_{\varepsilon, \lambda}$ such that $H(w_{\varepsilon, \lambda}) = \varepsilon$ for $\lambda$ large enough due to assumption \ref{item:existulam} which 
implies that $u_\lambda-\bar F(0)$ and $u_\lambda$ tends to $0$ and $1$, respectively. In fact we assume w.l.o.g.~that $\lambda$ is sufficiently close to one such that $u_\lambda-\lambda^b \leq \varepsilon$. Therefore 
$\bar w_{\lambda} \leq w_{\varepsilon, \lambda}$ as $H(w)$ is increasing and
$H(\bar w_\lambda) = u_\lambda - \lambda^b$.
By integrating \eqref{eq:H_prime_over_H_upper} from $\bar w_{\lambda}$ to $w_{\varepsilon,\lambda}$ we find:
\begin{align*}
\log\left( \frac{H(w_{\varepsilon,\lambda})}{H(\bar w_\lambda)} \right)
&=\int_{\bar w_\lambda}^{w_{\varepsilon, \lambda}} \frac{H'(u)}{H(u)} \, du\\
&\leq (w_{\varepsilon,\lambda}- \bar w_{\lambda}) \cdot (h_{\lambda}(u_\lambda - \lambda^b)-1).
\end{align*}
Dividing both sides by $-\log(1-\lambda)$ and taking the limit $\lambda \rightarrow 1^-$ we obtain:
\begin{align*}
\lim_{\lambda \rightarrow 1^-} - \frac{\log(\varepsilon) - \log(u_{\lambda} - \lambda^b)}{\log(1-\lambda)}
&= \lim_{\lambda \rightarrow 1^-} - \frac{\log(H(w_{\varepsilon,\lambda}))-\log(u_{\lambda}-\lambda^b)}{\log(1-\lambda)}\\
& \leq \lim_{\lambda \rightarrow 1^-} -\left(
\frac{w_{\varepsilon,\lambda}}{\log(1-\lambda)} \cdot \left( h_{\lambda}(u_{\lambda}-\lambda^b) - 1) \right)
\right),
\end{align*}
as $\bar w_\lambda$ is bounded by $\bar w$ due to assumption \ref{item:ODE_only_new_req}.
Applying assumptions \ref{item:lim_hlam} and \ref{item:lim_ulam} we obtain:
$$
\frac{B}{A-1} \leq  \lim_{\lambda \rightarrow 1^-} -\frac{w_{\varepsilon,\lambda}}{\log(1-\lambda)}.
$$
For any $w \leq w_{\varepsilon, \lambda}$ we have $1-\varepsilon \leq u_{\lambda} - \varepsilon 
= \bar F(w_{\varepsilon, \lambda}) \leq \bar F(w)$. It follows that:
\begin{align*}
(1-\varepsilon) \frac{B}{A-1}
&\leq \lim_{\lambda \rightarrow 1^-} - \frac{\int_{0}^{w_{\varepsilon,\lambda}} (1-\varepsilon) \, du}{\log(1-\lambda)} \leq \lim_{\lambda \rightarrow 1^-} - \frac{\int_0^\infty \bar F(w)\, dw}{\log(1-\lambda)}.
\end{align*}
This shows one inequality by letting $\varepsilon\rightarrow 0^+$. To show the other we first note that for any $w \in (\bar w_{\lambda}, w_{\varepsilon, \lambda})$ we have 
$u_\lambda-\lambda^b \leq H(w) \leq \varepsilon$ and therefore also:
$$
\frac{H'(w)}{H(w)}
=
h_{\lambda}(H(w)) - 1
\geq
h_{\lambda}(\varepsilon) - 1.
$$
Integrating both sides from $\bar w_\lambda$ to $w_{\varepsilon,\lambda}$ we find:
$$
\log \left( \frac{H(w_{\varepsilon,\lambda})}{H(\bar w_{\lambda})} \right) \geq (w_{\varepsilon, \lambda} - \bar w_{\lambda}) (h_\lambda(\varepsilon) - 1).
$$
Dividing both sides by $-\log(1-\lambda)$ and taking the limit of $\lambda \rightarrow 1^-$, this implies that we have (also use assumption \ref{item:lim_ulam}):
\begin{align*}
\lim_{\lambda \rightarrow 1^-} - \frac{w_{\varepsilon,\lambda}}{\log(1-\lambda)} (h_\lambda(\varepsilon) - 1)
&\leq \lim_{\lambda \rightarrow 1^-} - 
\left( \frac{\log(\varepsilon) - \log(u_{\lambda} - \lambda^b)}{\log(1-\lambda)} \right)=B.
\end{align*}
Note that we have:
\begin{align}
\int_{w_{\varepsilon, \lambda}}^\infty \bar F(u) \, du = \int_{w_{\varepsilon, \lambda}}^\infty \frac{\bar F(u)}{\bar F'(u)} \, d\bar F(u) = \int_0^{u_{\lambda} - \varepsilon} \frac{1}{1-\frac{T_\lambda(x)}{x} } \, dx, \label{eq:int_end_ODE}
\end{align}
assuming that $\lambda$ is sufficiently close to one, we find from assumption \ref{item:Tu_smaller_u} that \eqref{eq:int_end_ODE} is bounded by $$(u_{\lambda} - \varepsilon)^2/((u_{\lambda} - \varepsilon)-T_\lambda(u_\lambda-\varepsilon)),$$ which can be bounded uniformly in $\lambda$. This allows us to obtain:
\begin{align*}
\lim_{\lambda \rightarrow 1^-} - \frac{\int_0^\infty \bar F(w) \, dw}{\log(1-\lambda)}
&\leq \lim_{\lambda \rightarrow 1^-} -\frac{\int_0^{w_{\varepsilon,\lambda}}1 \, du}{\log(1-\lambda)} =\lim_{\lambda \rightarrow 1^-} - \frac{w_{\varepsilon, \lambda}}{\log(1-\lambda)} \leq \lim_{\lambda \rightarrow 1^-} \frac{B}{h_\lambda(\varepsilon)-1}.
\end{align*}
Taking the limit $\varepsilon \rightarrow 0^+$ and applying assumption \ref{item:lim_heps}, this completes the proof.
\end{proof}

\section{LL($d,K$)} \label{sec:LLdK}
We consider the LL($d,K$) model, that is, with rate $\lambda/K$ a group of $K$ i.i.d.~jobs which have an exponential size with mean $1$ arrive to the $K$ least loaded servers amongst $d$ randomly selected servers. In this section we give the detailed proof that \eqref{eq:lim_LLdK} is indeed valid for LL($d, K$). It is shown in \cite{hellemans2019performance} that the ccdf of the equilibrium workload distribution $\bar F(w)$ satisfies the ODE \eqref{eq:ODE} with $T_\lambda(u)$ given by \eqref{eq:LLdK_Tlam}.

The fact that $\bar F(w)$ is decreasing and $\int_0^\infty \bar F(w) \, dw < \infty$
is a consequence of the following result:
\begin{proposition}\label{prop:SQdK_recursion_correct}
For any $\lambda, u \in (0,1)$ with $T_\lambda$ defined as in \eqref{eq:LLdK_Tlam} we have $T_{\lambda}(u) \leq \lambda u$. In particular it follows that assumption \ref{item:ODE_only_new_req} holds with $\bar w = b-1$.
\end{proposition}
\begin{proof}
See Appendix \ref{app:proof_SQdK_recursion_correct}.
\end{proof}
We show that assumption \ref{item:Tu_smaller_u} holds,
note that the first bullet is immediate from the previous result.
\begin{lemma}\label{lem:Tuoveru_increasing_SQdK}
Let $T_\lambda$ be defined as in \eqref{eq:LLdK_Tlam} and let $u \in (0,1)$, the following inequality holds:
\begin{equation}\label{eq:Tuoveru_increasing}
\left(\frac{T_\lambda(u)}{u}\right)' > 0.
\end{equation}
\begin{proof}
See Appendix \ref{proof_lem:Tuoveru_increasing_SQdK}.
\end{proof}
\end{lemma}

We have the following elementary Lemma:
\begin{lemma}\label{lem:fplusag}
Let $f: [0,1] \rightarrow \mathbb{R}$ and $g:[0,1]\rightarrow [0,\infty)$ be continuous differentiable functions and let $h_a(x)=f(x)+ag(x)$ for $a \in [0,1]$. If $f(0)=0$ and $f'(0)<0$ then there exists a value $a_0>0$ such that for all $a \in [0,a_0]$ the function $h_a(x)$ has a root in $[0,1]$. Moreover if we let $x_a = \min\{x \in [0,1] \mid h_a(x)=0\}$ we have $\lim_{a\rightarrow 0^+}x_a = 0$.
\end{lemma}
\begin{proof}
See Appendix \ref{app:lem:fplusag}.
\end{proof}
The most difficult assumption to verify for the LL($d,K$) policy is assumption \ref{item:h_decreasing}, therefore we first validate the other remaining assumptions (note that we already verified assumptions \ref{item:Tu_smaller_u}, \ref{item:ODE_only_new_req}, \ref{item:lim_hlam} and \ref{item:lim_ulam}). We have:
\begin{lemma}\label{lem:SQdk1}
Let $1\leq K < d$ be fixed. There exists a $\bar \lambda < 1$ such that for all $\lambda \in (\bar \lambda, 1)$, the equation $T_\lambda(u)=u$ with $T_\lambda$ defined as in \eqref{eq:LLdK_Tlam} has a solution on $(1,\infty)$. Moreover, if we let $u_{\lambda}$ denote the minimal solution in $[1,\infty)$ for $\lambda \in (\bar \lambda,1)$ we have:
\begin{enumerate}[label=(\alph*), leftmargin=*]
\item $\lim_{\lambda \rightarrow 1^-} u_{\lambda} = 1$, therefore assumption \ref{item:existulam} holds. \label{enum:SQdK1}
\item We have: $\lim_{\varepsilon \rightarrow 0^+} \lim_{\lambda \rightarrow 1^-} h_\lambda(\varepsilon) = A$. Therefore assumption \ref{item:lim_heps} holds. \label{enum:ulamdK8}
\end{enumerate}
\end{lemma}
\begin{proof}
The proof can be found in Appendix \ref{app:lem:SQdk1}.
\end{proof}

To show assumption \ref{item:h_decreasing} we first need to do some additional work, in particular we compute $\lim_{\lambda \rightarrow 1^-} u_{\lambda}^{(n)}$ for $n=1,\dots,K+1$ (c.f.~Lemma \ref{lem:Dulamn}). To show this result, we employ the Fa\`a di Bruno formula which states that for functions $f$ and $g$ we have:
\begin{align*}
(f\circ g)^{(n)}(x)
&=
\sum_{k=1}^n f^{(k)}(g(x)) \cdot B_{n,k}(g'(x),g''(x),\dots, g^{(n-k+1)}(x)),
\end{align*}
where $B_{n,k}$ denotes the exponential Bell polynomial defined as:
\begin{align*}
B_{n,k}(x_1,\dots,x_{n-k+1})
&=
\sum \frac{n!}{j_1!j_2! \cdot \dots \cdot j_{n-k+1}!} \left( \frac{x_1}{1!} \right)^{j_1} \cdot \dots \cdot \left( \frac{x_{n-k+1}}{(n-k+1)!} \right)	^{j_{n-k+1}}.
\end{align*}
Here the sum is taken over all non-negative integers $j_1,\dots, j_{n-k+1}$ which satisfy:
\begin{align*}
k =j_1+\dots+j_{n-k+1}
\qquad \mbox{ and } \qquad n =j_1+2j_2+\dots+(n-k+1)j_{n-k+1}.
\end{align*}
Furthermore we employ the fact that:
\begin{align}\label{eq:Lah}
B_{n,k}(1!,\dots,(n-k+1)!)=\frac{n!}{k!} \binom{n-1}{k-1},
\end{align}
which are known as the Lah numbers. We are now able to show:

\begin{lemma}\label{lem:Dulamn}
For any $d,K$ we have:
\begin{align}\label{eq:Dulam_lowerK}
\lim_{\lambda \rightarrow 1^-} u_\lambda^{(n)} &= (-1)^n n! \frac{d^{n-1} K }{(d-K)^n}
\end{align}
for $1 \leq n \leq K $ and 
\begin{align}
\lim_{\lambda \rightarrow 1^-} u_{\lambda}^{(K+1)} &= (-1)^{K+1} (K+1)! \frac{d^K K}{(d-K)^{K+1}}
- \frac{d!}{(d-K)!} \left( \frac{K}{d-K} \right)^{K+1}\label{eq:Dulam_Kplus1}
\end{align}
\end{lemma}
\begin{proof}
We give a sketch of the proof, for the complete proof we refer to Appendix \ref{app:proof:lem:Dulamn}.
We define $\Theta(u)=\frac{1}{\lambda}\frac{T_\lambda(u)}{u}$, one can compute $\Theta^{(n)}(u)$ by induction and taking the limit of $u \rightarrow 1^-$, it is possible to compute:
\begin{align}
\Theta^{(n)}(1) &= (-1)^{n+1} n! \frac{d-K}{K} & \mbox{for } 1 \leq n \leq K ,  \label{eq:Dtheta_lowerK_paper}\\
\Theta^{(K+1)}(1) &= (-1)^{K+1} \frac{d! - (d-K)! (K+1)!}{(d-K)!} \frac{d-K}{K}. \label{eq:Dtheta_Kplus1_paper}
\end{align}

We now continue by induction to show (\ref{eq:Dulam_lowerK}-\ref{eq:Dulam_Kplus1}). For the case $n=1$ we note that from $u_{\lambda}=T_\lambda(u_{\lambda})$ it follows that:
\begin{align}
u_{\lambda}' &= 
\frac{1}{K} \sum_{j=0}^{K-1} (K-j) \binom{d}{j} u_\lambda ^{d-j} (1-u_\lambda)^j + \frac{\lambda}{K} \sum_{j=0}^{K-1} (K-j) (d-j) \binom{d}{j} u_\lambda^{d-j-1} (1-u_\lambda)^j u_\lambda' \nonumber\\
& -\frac{\lambda}{K} \sum_{j=1}^{K-1} (K-j) j \binom{d}{j} u_\lambda^{d-j} (1-u_\lambda)^{j-1} u_\lambda'. \label{eq:Dulam_LLdk_paper}
\end{align}
Taking the limit $\lambda \rightarrow  1^-$ we obtain $\lim_{\lambda \rightarrow 1^-} u_{\lambda}' = 1 + \frac{d}{K} \lim_{\lambda \rightarrow 1^-} u_{\lambda}'$ yielding \eqref{eq:Dulam_lowerK} with $n=1$. Let $2 \leq n \leq K+1$ and note that $u_\lambda= T_\lambda(u_\lambda)$ may be reformulated as $1=\lambda \Theta(u_\lambda)$. By differentiating both sides $n\geq 2$ times, it follows that we have:
\begin{align}\label{eq:1=lamTheta_paper}
0= n \left( \frac{\partial}{\partial \lambda} \right)^{n-1} \Theta(u_\lambda) + \lambda \left( \frac{\partial}{\partial \lambda} \right)^{n} \Theta(u_\lambda).
\end{align}
It follows from the Fa\`a di Bruno formula that:
\begin{equation}\label{eq:bruno_paper}
\left(\frac{\partial}{\partial \lambda}\right)^n \Theta (u_\lambda)
=
\sum_{k=1}^n \Theta^{(k)}(u_\lambda) B_{n,k}(u_\lambda',\dots, u_{\lambda}^{(n-k+1)})
\end{equation}
where $B_{n,k}$ denotes the exponential Bell polynomial. We have:
$$
B_{n,1}(u_\lambda',\dots,u_\lambda^{(n)})=u_\lambda^{(n)}
$$
(as $j_1 = \ldots = j_{n-1} = 0$ and $j_n=1$) and for $k >1$ using induction and \eqref{eq:Lah} one can show that:
\begin{align*}
\lim_{\lambda \rightarrow 1^-} B_{n,k}(u_\lambda',\dots,u_{\lambda}^{(n-k+1)})
&=
\frac{n!}{k!} \binom{n-1}{k-1} (-1)^n \frac{d^{n-k} K^k}{(d-K)^n}.
\\
\lim_{\lambda \rightarrow 1^-} B_{n-1,k}(u_\lambda',\dots,u_{\lambda}^{(n-k)})
&=
\frac{(n-1)!}{k!} \binom{n-2}{k-1} (-1)^{n-1} \frac{d^{n-k-1} K^k}{(d-K)^{n-1}}.
\end{align*}
Therefore, by combining \eqref{eq:Dtheta_lowerK_paper}, \eqref{eq:1=lamTheta_paper}, \eqref{eq:bruno_paper} and using Pascal's triangle one can compute that for $n \leq K+1$:
\begin{align*}
\lim_{\lambda\rightarrow 1^-} u_\lambda^{(n)}
&=
(-1)^{n+1} \left( \frac{K}{d-K} \right)^{n+1} \Theta^{(n)}(1) -n! \left( \frac{K}{d-K} \right)^n + n! (-1)^n \frac{d^{n-1} K}{(d-K)^n},
\end{align*}
for $n \leq K+1$.
Plugging in 
 \eqref{eq:Dtheta_lowerK_paper} and \eqref{eq:Dtheta_Kplus1_paper} yields
  \eqref{eq:Dulam_lowerK} and \eqref{eq:Dulam_Kplus1}, respectively.
\end{proof}

We are now able to show that assumption \ref{item:h_decreasing} indeed holds:

\begin{lemma}\label{lem:LLdk2}
Let $1\leq K < d$ be fixed, there exists a $\tilde \lambda< 1$ and $b \in \mathbb{N}$ (independent of $\lambda$) such that the function $h_{\lambda}(x)$ defined as in \eqref{eq:def_hlam} with $T_{\lambda}$ as in \eqref{eq:LLdK_Tlam} is decreasing as a function of $x\in[u_\lambda-\lambda^b,u_\lambda]$ for all $\lambda \in (\tilde \lambda,1)$.
\end{lemma}
\begin{proof}
Here we give a sketch of the proof, the full proof can be found in Appendix \ref{app:proof:lem:LLdk2}.
Throughout, we assume that $\bar \lambda < \lambda < 1$ with $\bar \lambda$ as in Lemma \ref{lem:SQdk1}. We show there is some $\tilde{\lambda} \geq \bar \lambda$ and $b \in \mathbb{N}$ which does not depend on the value of $\lambda$ for which $h_{\lambda}(x)$ is decreasing on $[u_\lambda - \lambda^b, u_{\lambda}]$ for all $\lambda \in (\tilde \lambda, 1)$.

We define $\zeta_\lambda(x)=Kx^2h_\lambda'(x)$ and note that it suffices to show that $\zeta_\lambda(x) \leq 0$ for $\lambda$ sufficiently close to one. To this end one can show that:
\begin{equation} \label{eq:DzetaLam_paper}
\zeta_\lambda'(x)=-\lambda (d-K) \binom{d}{K-1} (d-K+1) (1-u_\lambda+x)^{K-1} (u_\lambda-x)^{d-K-1} x.
\end{equation}
This is obviously negative for all $x \in [u_\lambda-1,u_\lambda]$. It thus suffices to show that we can find a value $b \in \mathbb{N}$ such that $\zeta_\lambda(u_\lambda-\lambda^b) \leq 0$.
Let us denote $\Theta(u) = \frac{1}{\lambda} \frac{T_\lambda(u)}{u}$, we can show that:
\begin{align}
\lambda^b \zeta_\lambda\left(u_\lambda-\lambda^b\right)
&= -\lambda K u_\lambda \lambda^b (\Theta(u_\lambda) - \Theta(\lambda^b))
 + \lambda (d-K) (u_\lambda - \lambda^b) \sum_{j=0}^{K-1} \binom{d}{j} \lambda^{b(d-j)} (1-\lambda^b)^j. \label{eq:lambzetaulamminlamb_paper}
\end{align}
For now let us focus on the case $K=d-1$. For this case we have:
\begin{equation}\label{eq:DnTheta_paper}
\Theta^{(n)}(u)=(-1)^{n+1} n! \frac{1}{d-1} \sum_{j=0}^{d-n-1} \binom{d}{j} u^{d-j-n-1} (1-u)^j,
\end{equation}

employing the Taylor expansion of $\Theta(u_\lambda)$ at $\lambda^b$, we find that \eqref{eq:lambzetaulamminlamb_paper} can be written as:
\begin{align}
\lambda^b \zeta_\lambda(u_\lambda-\lambda^b)
&=
\lambda \bigg[ (u_\lambda - \lambda^b) \sum_{j=0}^{d-2} \binom{d}{j} \lambda^{b(d-j)} (1-\lambda^b)^j
- u_\lambda \lambda^b \sum_{n=1}^{d-1} (d-1) \Theta^{(n)}(\lambda^b) \frac{(u_\lambda-\lambda^b)^n}{n!}\bigg].\label{eq:taylor_exp_paper}
\end{align}
Combining \eqref{eq:DnTheta_paper} and \eqref{eq:taylor_exp_paper} we are able to compute:
\begin{align*}
\lim_{\lambda \rightarrow 1^-} \frac{\lambda^b}{(u_\lambda-\lambda^b)^d} \zeta_\lambda(u_\lambda-\lambda^b) 
&=-d \left( \frac{b}{b+d-1} \right)^{d-1} + \left(\frac{b}{b+d-1}\right)^d + (-1)^{1+d} \left( \frac{d-1}{b+d-1} \right)^d,
\end{align*}
which converges to $1-d \leq 0$ as $b$ tends to infinity. This proves Lemma \ref{lem:LLdk2}
for $K=d-1$.

Fix $K$ and let $d \geq K+1$ be variable. Let $(K_1,d_1)$ and $(K_2,d_2)$ be arbitrary (with $K_i < d_i$), denote by $_i u _\lambda$ the fixed point associated to $(K_i,d_i)$ and $_i\zeta_\lambda$ the associated $\zeta_\lambda$ function. We can show the following inequalities:
\begin{align}
_2 \zeta _\lambda'(x) & \leq\  _1\zeta _\lambda'(x+\ _1u_\lambda -\ _2u_\lambda) & \mbox{for }  x \in [\ _2u_\lambda - 1,\ _2u_\lambda-\lambda^b] \label{eq:ineq_zeta_prime_paper}\\
_2\zeta_\lambda(_2 u _\lambda - 1) &\leq\ _1\zeta _\lambda(_1 u _\lambda - 1), & \label{eq:ineq_zeta_ulam_min_1_paper}
\end{align}
in case we have:
\begin{enumerate}[label=(\roman*),topsep=0pt]
\item \label{enum:Keven} $K$ is even, $(K_1,d_1)=(K,d)$ and $(K_2,d_2)=(K,d+1)$ or
\item \label{enum:Kodd} $K$ is odd, $(K_1,d_1)=(K+1,d+1)$ and $(K_2,d_2)=(K,d)$.
\end{enumerate}
If (\ref{eq:ineq_zeta_prime_paper}-\ref{eq:ineq_zeta_ulam_min_1_paper}) hold, we find that:
\begin{align*}
_2\zeta_\lambda(_2u_\lambda-\lambda^b)
&=\ _2\zeta_\lambda( _2u_\lambda-1)+\int_{_2 u_\lambda-1}^{_2 u_\lambda-\lambda^b} \ _2 \zeta_\lambda'(x)\, dx\\
&\leq\ _1 \zeta_\lambda(_1u_\lambda-1)+\int_{_2u_\lambda-1}^{_2u_\lambda-\lambda^b} \ _1 \zeta_\lambda'(x+\ _1u_{\lambda} -\ _2u_{\lambda})\, dx
=\ _1 \zeta_\lambda(_1 u_{\lambda} - \lambda^b)
\end{align*}
This shows that if $\ _1 \zeta_\lambda(\ _1 u_{\lambda} - \lambda^b) \leq 0$, then also $_2 \zeta_\lambda(_2u_\lambda-\lambda^b) \leq 0$. Applying \ref{enum:Keven} then concludes the proof for $K$ even as we already established the result for $K=d-1$. Having shown the result for $K$
even then implies that the result also holds for $K$ odd by applying \ref{enum:Kodd}.
\end{proof}

It follows that assumption \ref{item:h_decreasing} indeed holds, therefore we may apply Corollary \ref{cor:gen_result} to show \eqref{eq:lim_LLdK}.

\section{LL($d_1,\dots, d_n, p_1,\dots, p_n$)} \label{sec:l1ld}
In this section, we consider a policy which, with probability $p_i$, sends an incoming job to the queue with the least amount of work left out of $d_i$ randomly selected servers. Throughout, we assume that $\sum_{i}p_i=1$, $d_i\geq 1$ and $\sum_i p_i d_i > 1$. We assume jobs are exponentially distributed with mean one and the arrival rate is equal to $\lambda \in (0,1)$. Let $\bar F(w)$ denote the probability that a queue in the mean field limit has $w$ or more work. We find that $\bar F(w$) can be found as the solution of a simple ODE. The proof of Proposition \ref{prop:LLd1dn_F} combines the main idea of the proofs of Theorem 4.1 and Theorem 5.1 in \cite{hellemans2018power} and the result of Theorem 5.2 in \cite{hellemans2019workload}.
\begin{proposition}\label{prop:LLd1dn_F}
The ccdf of the workload distribution for LL($d_1,\dots,d_n,p_1,\dots,p_n$) satisfies ODE \eqref{eq:ODE} with $T_\lambda(u)$ defined as in \eqref{eq:Fbar_LLd1dn}. 
\end{proposition}
\begin{proof}
The proof can be found in Appendix \ref{app:LLd1dn_F}.
\end{proof}

In Section \ref{sec:LLd1..dn} we show that one may apply Corollary \ref{cor:gen_result}. In Section \ref{sec:impact_di_pi} we answer the question whether a certain choice of $p_i$ and $d_i$ is better than another, moreover we show which choice of $p_i$ and $d_i$ is optimal given a fixed number of probes which are allowed per arrival (that is $\sum_i p_i d_i$ fixed). Lastly in Section \ref{sec:LLdp} we show that when some jobs use the LL($d$) policy while other jobs are assigned arbitrarily, the ODE can be solved explicitly and we use this result to give an alternative proof of our limit result.

\subsection{The Limit Result}\label{sec:LLd1..dn}
We first show that the requirements to apply Theorem \ref{thm:gen_result_ODE} (and therefore also Corollary \ref{cor:gen_result}) are satisfied with $\bar \lambda = 0$ in assumpton \ref{item:existulam}, $b=0$ in assumption \ref{item:h_decreasing} and $\bar w = 0$ in assumption \ref{item:ODE_only_new_req}.
\begin{lemma} \label{lem:requirements_SQd1dn}
For any $\lambda \in (0,1)$ the equation $u=T_\lambda(u)$ with $T_\lambda(u)=\sum_i p_i u^{d_i}$ as in Proposition \ref{prop:LLd1dn_F} has exactly one solution $u_{\lambda}$ on $(1,\infty)$ moreover this solution satisfies:
\begin{enumerate}[label=(\alph*),leftmargin=*]
\item $\lim_{\lambda \rightarrow 1^-} u_{\lambda} = 1$, in particular assumption \ref{item:existulam} holds. \label{enum:d1dn_1}
\item $T_\lambda(u) < u$, $(T_\lambda(u)/u)' \geq 0$ and
$\lim_{\lambda \rightarrow 1^-} T_\lambda(u)/u \leq \lambda  < 1$ in particular assumption \ref{item:Tu_smaller_u} holds.
\item The function $\xi_i(x)=\frac{u_\lambda^{d_i} - (u_\lambda - x)^{d_i}}{x}$ is decreasing on $[u_\lambda-1,u_\lambda]$. In particular assumption \ref{item:h_decreasing} holds with $b=0$
as $h_\lambda(x)=\lambda \sum_i p_i \xi_i(x)$. Moreover assumption \ref{item:ODE_only_new_req} holds with $\bar w=0$. \label{enum:d1dn_7}
\item We have $\lim_{\varepsilon \rightarrow 0^+} \lim_{\lambda \rightarrow 1^-} h_\lambda(\varepsilon)= \sum_{i=1}^n p_i d_i$, in particular assumption \ref{item:lim_heps} holds. \label{enum:lim_heps_d1dn}
\end{enumerate}
\end{lemma}
\begin{proof}
See Appendix \ref{app:lem:requirements_SQd1dn}.
\end{proof}

Let $W_\lambda^{(i)}$ and $R_\lambda^{(i)}$ denote the waiting and response time for a job which is assigned to the server with the Least Loaded queue amongst $d_i$ randomly selected servers and $W_\lambda=\sum_i p_i W_\lambda^{(i)}$ and $R_\lambda=\sum_i p_i R_\lambda^{(i)}$ the waiting and response time for $LL(d_1,\dots,d_n,p_1,\dots,p_n)$. We obtain the following limits:
\begin{theorem} \label{thm:LLd1dn}
We have for all $i$:
\begin{align*}
\lim_{\lambda\rightarrow 1^-} - \frac{\E[R_\lambda^{(i)}]}{\log(1-\lambda)}
= \lim_{\lambda\rightarrow 1^-} - \frac{\E[R_\lambda]}{\log(1-\lambda)}
= \lim_{\lambda\rightarrow 1^-} - \frac{\E[W_\lambda^{(i)}]}{\log(1-\lambda)}
= \lim_{\lambda\rightarrow 1^-} - \frac{\E[W_\lambda]}{\log(1-\lambda)}
= \frac{1}{\sum_{i=1}^n p_i d_i - 1}.
\end{align*}
\end{theorem}
\begin{proof}
The last equality follows from Corollary \ref{cor:gen_result} as we verified all assumptions in Lemma \ref{lem:requirements_SQd1dn} and Section \ref{sec:computation}.
Moreover we have for any $i$:
$$
\E[W_\lambda^{(i)}]
=
\int_0^\infty \bar F(w)^{d_i} \, dw
\leq \int_0^\infty \bar F(w) \, dw.
$$
Therefore we have in the limit $\lim_{\lambda\rightarrow 1^-} - \frac{\E[W_\lambda^{(i)}]}{\log(1-\lambda)} \leq \frac{1}{\sum_i p_i d_i - 1}$ and
$$
\frac{1}{\sum_i p_i d_i - 1}
=\lim_{\lambda\rightarrow 1^-} -\frac{\E[W_\lambda]}{\log(1-\lambda)}
=
\lim_{\lambda\rightarrow 1^-} -\frac{\sum_i p_i \E[W_\lambda^{(i)}]}{\log(1-\lambda)}
=
\sum_i p_i \lim_{\lambda\rightarrow 1^-} -\frac{\E[W_\lambda^{(i)}]}{\log(1-\lambda)}.
$$
This concludes the proof as $\sum_i p_i = 1$ and $p_i \geq 0$ for all $i$.
\end{proof}

\subsection{The impact of $d_i$ and $p_i$} \label{sec:impact_di_pi}
Let $p_1,\dots,p_n$ and $d_1,\dots,d_n$ be arbitrary. As the function $\varphi_u(x)=u^x$ is a convex function for any $u\in [0,1]$, we find that for all $u \in [0,1]$ we have $u^{\sum_{i=1}^n p_i d_i} \leq \sum_{i=1}^n p_i u^{d_i}$. From \eqref{eq:Fbar_LLd1dn} it therefore follows that for any arrival rate $\lambda \in (0,1)$ and fixed $\sum_{i=1}^n p_i d_i \in \mathbb{N}$ the optimal policy is LL($\sum_{i=1}^n p_i d_i$). On the other hand it follows from Theorem \ref{thm:LLd1dn} that as $\lambda$ tends to one, the choice of $p_i$ and $d_i$ does not affect the mean waiting time. In this section we take a closer look at which choices of $p_i$ and $d_i$ yield lower waiting times.

More specifically, one may wonder whether LL($a_1,\dots,a_n,p_1,\dots,p_n$) outperforms another policy\newline LL($b_1,\dots,b_n,q_1,\dots,q_n$). Moreover, given a maximal amount of average choice $\sum_{i=1}^n p_i d_i$ on job arrival, what is the optimal choice of $p_i \in [0,1]$ and $d_i \in \mathbb{N}$? To answer these questions we introduce the concept of Majorization with weights which is presented in \cite{marshall1979inequalities} (Chapter IV, Section 14 A). Specifically the following result is shown (originally introduced in \cite{blackwell1951range}, but a more comprehensive proof can be found in \cite{borcea2007equilibrium}):
\begin{proposition}\label{prop:majorization}
Let $p=(p_1,\dots,p_n)$ and $q=(q_1,\dots,q_m)$ be fixed vectors with nonnegative components such that $\sum_{i=1}^n p_i = \sum_{i=1}^m q_i = 1$. For $x \in \R^n, y \in \R^m$ the following are equivalent
\begin{enumerate}
\item For all convex functions $\varphi:\R\rightarrow \R$ we have $\sum_{i=1}^n p_i \varphi(x_i) \leq \sum_{i=1}^m q_i \varphi(y_i)$.
\item There exists an $m \times n$ matrix $A = (a_{ij})$ which satisfies $a_{ij} \geq 0, eA=e$ (with $e=(1,\dots,1)$), $A p^T = q^T$ (with $p^T$ the transpose of $p$) and $x=yA$. \label{enum:major_2}
\end{enumerate}
\end{proposition}

As a consequence of Proposition \ref{prop:majorization} we say that $(p,a)$ is majorized by $(q,b)$ and write $(p,a) \preceq (q,b)$ if and only if (1) or (2) in Proposition \ref{prop:majorization} holds. The interpretation is that $(q,b)$ is more scattered than $(p,a)$.  This yields a method for comparing policies as $(p,a) \preceq (q,b)$ also implies that 
the workload distribution $\bar F(w)$ of LL($p,a$) is upper bounded by
the workload distribution of LL($q,b$).

Despite the fact that given a budget $\bar d = \sum_i p_i d_i$, the optimal policy is simply LL($\bar d$), we may have $\bar d \notin \mathbb{N}$. In this case we simply use LL(($p$,$1-p$), ($\lfloor \bar d \rfloor$, $\lceil \bar d \rceil$)) for an appropriate $p\in [0,1]$. We show that this is indeed the optimal choice (here $\lfloor \bar d \rfloor$ denotes the floor and $\lceil \bar d \rceil$ denotes the ceil of $\bar d$).
\begin{theorem}
Let $p=(p_1,\dots,p_n)$ with $\sum_{i=1}^n p_i=1, p_i \geq 0$ and $d=(d_1,\dots,d_n)$ with $d_i \in \mathbb{N}$. If we let $\bar d=\sum_{i=1}^n p_i d_i$ and $q=(q_1,q_2)$ s.t.~$q_1+q_2=1$ and $q_1 \lfloor \bar d \rfloor + q_2 \lceil \bar d \rceil = \bar d$ then $(q,(\lfloor \bar d \rfloor, \lceil \bar d \rceil)) \preceq (p,d)$.
\end{theorem}
\begin{proof}
We show that Proposition \ref{prop:majorization}, \eqref{enum:major_2} holds, to this end we let $A= (a_{ij}) \in \R^{n,2}$. From $Aq^T = p$ it follows that for all $j$:
\begin{equation}\label{eq:proof_majorization}
a_{1j} = \frac{p_j - q_2 a_{2j}}{q_1}.
\end{equation}
It is not hard to see that one can indeed choose $0 \leq a_{2j} \leq \frac{p_j}{q_2}$ such  that $\sum_{j=1}^n a_{2j} d_j = \lceil \bar d \rceil$ and $\sum_j a_{2j} = 1$, it then automatically follows from \eqref{eq:proof_majorization} that also $\sum_j a_{1j} =1$. Moreover it follows that:
\begin{align*}
\sum_{j=1}^n a_{1j} d_j
&= \frac{\bar d}{q_1} - \frac{q_2}{q_1} \lceil \bar d \rceil = \lfloor \bar d \rfloor.
\end{align*}
This completes the proof.
\end{proof}

\subsection{LL($d,p$)}\label{sec:LLdp}
We take a closer look at the particular case where $n=2$ and $d_2=1$, we denote $p=p_1$ and thus $p_2=1-p$. We write LL($d,p$) as a shorthand for LL($d_1,d_2,p_1,p_2$). In practice this policy can be viewed as having two arrival streams : one at each server individually, at rate $\lambda (1-p)$ for which 
there is no load balancing and a second at rate $\lambda p N$ which is distributed using the LL($d$) load balancing policy. It turns out that (as for LL($d$) in \cite{hellemans2018power}), this policy has a closed form solution for the ccdf of the workload distribution:
\begin{proposition}
The equilibrium workload distribution for the LL($d,p$) policy with exponential job sizes of mean one is given by \eqref{eq:exp_LLdp_WLdist}.
\end{proposition}
\begin{proof}
In this case, the ODE defined in \eqref{eq:Fbar_LLd1dn} reduces to:
\begin{equation}
\bar F'(w)= \lambda \left(p\bar F(w)^d + (1-p) \bar F(w) - \frac{\bar F(w)}{\lambda} \right). \label{eq:proof_explicit1}
\end{equation}
This is an autonomous ODE, we find that it can be solved explicitly simply by writing it as:
$$
\frac{d \bar F(w)}{p\bar F(w)^d + (1-p) \bar F(w) - \frac{\bar F(w)}{\lambda}} = \lambda\, dw,
$$
integrating and rewriting in function of $\bar F(w)$ yields \eqref{eq:exp_LLdp_WLdist}.
\end{proof}
From this we find:
\begin{proposition} \label{prop:WL_LLd}
The mean queue length for the LL($d,p$) policy with exponential job sizes of mean one is given by:
\begin{equation}\label{eq:WL_LLd}
\E[Q_\lambda]=\frac{\lambda}{1-(1-p)\lambda} \sum_{n=0}^\infty \frac{1}{1+n(d-1)} \left( \frac{p \lambda^d}{1-(1-p)\lambda}\right)^n.
\end{equation}
In particular for $d=2$ we have:
$$
\E[Q_\lambda] = - \frac{\log\left( 1 - \frac{p \lambda^2}{1-(1-p) \lambda} \right)}{p \lambda}.
$$
\end{proposition}
\begin{proof}
The proof can be found in Appendix \ref{app:prop:WL_LLd}.
\end{proof}
The mean waiting time is now given by $\E[W_\lambda] = \frac{\E[Q_\lambda]}{\lambda} - 1$, using this all results involving mean queue length can easily be adapted to mean waiting time. We find a simple lower and upper bound for the mean queue length:
\begin{proposition}\label{prop:approxWL_LLd}
We have:
$$
\tilde Q_\lambda - \frac{\lambda^{d+1}}{p (d-1)^2(1-(1-p)\lambda)^2} \frac{\pi^2}{6} \leq \E[Q_\lambda] \leq \tilde Q_\lambda,
$$
with:
\begin{align}
\tilde{Q}_\lambda
&=
\frac{\lambda}{1-(1-p)\lambda} \cdot \left( 1 + \frac{1}{d-1} \log\left(\frac{1-(1-p)\lambda}{1-(1-p)\lambda-p\lambda^d}\right) \right). \label{eq:WLLdtilde}
\end{align}
\end{proposition}
\begin{proof}
Throughout this proof we denote $z = \frac{p \lambda^d}{1-(1-p)\lambda}$, where $z < 1$ for $\lambda < 1$.
We first note that as $\log(1/(1-z))=\sum_{n\geq 1} z^n/n$, $\tilde{Q}_\lambda$ can be written as
$$
\tilde{Q}_\lambda
=
\frac{\lambda}{1-(1-p)\lambda} \left(1+ \sum_{n=1}^\infty \frac{z^n}{n(d-1)} \right).
$$
From this it is obvious that $\E[Q_\lambda] \leq \tilde{Q}_\lambda$. Furthermore we find:
\begin{align*}
\tilde Q_\lambda - \E[Q_\lambda]
&=
\frac{\lambda}{(d-1)(1-(1-p)\lambda)} \sum_{n=1}^{\infty} \frac{z^n}{(1+n(d-1))n}\\
&\leq \frac{\lambda}{(d-1)^2(1-(1-p)\lambda)} \sum_{n=1}^\infty \frac{z^n}{n^2}
\leq \frac{\lambda z \pi^2}{6(d-1)^2(1-(1-p)\lambda)},
\end{align*}
as $\sum_{n \geq 1} 1/n^2 = \pi^2/6$.
This concludes the proof.
\end{proof}

Similar bounds for  LL($d$), i.e.~when $p=1$ were not presented in \cite{hellemans2018power}.
Using these bounds we obtain an alternative proof for the result in Theorem \ref{thm:LLd1dn}
for the special case where $n=2$ and $d_2 = 1$. Indeed, a simple application of l'Hopital's rule yields that:
$$
\lim_{\lambda \rightarrow 1^-} - \frac{\E[\tilde{W}_\lambda]}{\log(1-\lambda)} 
= \lim_{\lambda \rightarrow 1^-} - \frac{\E[\tilde{Q}_\lambda]}{\log(1-\lambda)} 
= \frac{1}{p(d-1)}.
$$

\section{Low Load Limit} \label{sec:low_traffic}
From the result $\lim_{\lambda\rightarrow 1^-} -\frac{\E[W_\lambda]}{\log(1-\lambda)} = x$, one could argue that for a sufficiently high value of $\lambda$, we have $\E[W_\lambda] \approx - x \cdot \log(1-\lambda)$. Instead of taking the limit $\lambda \rightarrow 1^-$ in \eqref{eq:gen_limit_response}, one could also consider the limit $\lambda \rightarrow 0^+$. However, it is not hard to see that for all policies we considered this limit is simply equal to zero, yielding $\E[W_\lambda] \approx 0$ for sufficiently small values of $\lambda$ which is not all too useful as an approximation. 

We therefore introduce the value $p_\lambda$ which denotes the probability that a job is assigned to an idle queue. For the LL($d$) policy, it is not hard to see that any job is assigned to an idle queue with probability $p_\lambda = 1-\lambda^d$.

We suggest to use $-\log(p_\lambda)$ as the scaling factor instead of $-\log (1-\lambda)$
such that we also get a non-zero limit when $\lambda$ tends to zero, that is, we consider the fraction $-\frac{\E[W_\lambda]}{\log(p_\lambda)}$. We show that the limit 
for $\lambda$ tending to one is not altered with the new scaling factor. In Section \ref{sec:num_exp} we show that this new scaling can be used as a useful approximation of the expected waiting time when $\lambda$ is close to $0$ or $1$.

We now explain how to compute $p_\lambda$ and the associated limits for the policies considered in this work.
\subsection{LL($d$)}
For the LL($d$) policy we have $p_\lambda = 1-\lambda^d$. It follows that:
\begin{proposition} \label{prop:bounds_LLd}
For the LL($d$) policy with exponential job sizes of mean one we have:
\begin{equation}
\lim_{\lambda\rightarrow 0^+} -\frac{\E[W_\lambda]}{\log(1-\lambda^d)} = \frac{1}{d} \qquad \mbox{ and } \qquad \lim_{\lambda\rightarrow 1^-} -\frac{\E[W_\lambda]}{\log(1-\lambda^d)} = \frac{1}{d-1}.
\end{equation}
\end{proposition}
\begin{proof}
It is easy to see that $\lim_{\lambda \rightarrow 1^-} \frac{\log(1-\lambda)}{\log(1-\lambda^d)} = 1$, which shows that the limit $\frac{1}{d-1}$ when $\lambda$ tends to one  remains valid. For the low load limit we only need to consider the case where the dispatcher selects $d$ servers which all have exactly one job in their queue. The LL($d$) policy assigns the incoming job to the server which finishes its job first. Therefore we find that the mean waiting time for the LL($d$) policy (for $\lambda \approx 0$) can be approximated by $(1-p_\lambda)/d = \frac{\lambda^d}{d}$. Therefore the result follows from the fact that $\lim_{\lambda \rightarrow 0^+} - \frac{\lambda^d}{\log(1-\lambda^d)} = 1$.
\end{proof}

\subsection{LL($d, K$)}
For the LL($d,K$) policy we can compute $p_\lambda$ using the same ideas and we obtain the following result:
\begin{proposition} \label{prop:LLdK_lb}
For the LL($d, K$) policy with exponential job sizes with mean one, we find that:
\begin{equation}
\lim_{\lambda \rightarrow 0^+} - \frac{\E[W_\lambda]}{\log(p_\lambda)} = \frac{1}{d-K+1} \qquad \mbox{ and } \qquad \lim_{\lambda \rightarrow 1^-} - \frac{\E[W_\lambda]}{\log(p_\lambda)} = \frac{K}{d-K} 
\end{equation}
with $p_\lambda = 1 - \sum_{j=0}^{K-1} \frac{K-j}{K} \binom{d}{j}(1-\lambda)^j \lambda^{d-j}$.
\end{proposition}
\begin{proof}
We first compute the probability $p_\lambda$ that an arbitrary job is assigned to an idle server for the LL($d, K$) policy. It is not hard to see that $p_\lambda$ is given by:
$$
p_\lambda = \frac{\E[\min\{\mbox{number of idle servers selected}, K\}]}{K} = \frac{\sum_{j=0}^d \min\{j, K\} \binom{d}{j} (1-\lambda)^j \lambda^{d-j}}{K}.
$$
This simplifies to the given formula for $p_\lambda$ by using $1 = \sum_{j=0}^d \binom{d}{j} (1-\lambda)^j \lambda^{d-j}$.

The low load limit can be computed by noting that the only situation we need to look at is a probe which finds $K-1$ idle servers and $d-K+1$ servers which are processing a single job. For the job with non-zero waiting time, we find that the expected waiting time is given by the minimum of $d-K+1$ exponential jobs. One finds that the expected waiting time (for $\lambda \approx 0$) is given by $\lambda^{d-K+1} \frac{\binom{d}{K-1}}{K\cdot (d-K+1)}$. From this the limit for $\lambda$ tending to $0$ follows after computing:
\begin{align*}
\lim_{\lambda \rightarrow 1^-} - \frac{\lambda^{d-K+1} \frac{\binom{d}{K-1}}{K\cdot (d-K+1)}}{\log\left( 1 - \sum_{j=0}^{K-1} \frac{K-j}{K} \binom{d}{j} (1-\lambda)^j \lambda^{d-j} \right)}
&=
\lim_{\lambda \rightarrow 1^-} - \frac{\lambda^{d-K+1} \frac{\binom{d}{K-1}}{K\cdot (d-K+1)}}{\log\left( 1 - \frac{1}{K} \binom{d}{K-1} (1-\lambda)^{K-1} \lambda^{d-K+1} \right)}\\
&= \frac{1}{d-K+1}.
\end{align*}
For the limit of $\lambda$ tending to $1$, one simply needs to show that $\frac{\log(p_\lambda)}{\log(1-\lambda)}$ converges to one as $\lambda \rightarrow 1^-$, which follows by applying l'Hopital's rule.
\end{proof}

\subsection{LL($d_1,\dots,d_n,p_1,\dots,p_n$)}
For the LL($d_1,\dots,d_n,p_1,\dots,p_n$) policy we find that the probability of assigning a job to an empty server is equal to $p_\lambda = \sum_i p_i \cdot (1 - \lambda^{d_i}) = 1-\sum_i p_i \lambda^{d_i}$. We have the following result:
\begin{proposition} \label{prop:low_traffic_LLdK}
For the LL($d_1,\dots,d_n,p_1,\dots,p_n$) policy with exponential job sizes with mean one we have:
\begin{equation} \label{eq:low_traffic_LLd1dn}
\lim_{\lambda \rightarrow 0^+} - \frac{\E[W_\lambda]}{\log(p_\lambda)} = \frac{1}{d_j} \qquad \mbox{ and } \qquad \lim_{\lambda \rightarrow 1^-} - \frac{\E[W_\lambda]}{\log(p_\lambda)} = \frac{1}{\sum_i p_i d_i - 1}
\end{equation}
with $p_\lambda = 1 - \sum_i p_i \lambda ^{d_i}$ and $j = \min\{ i \mid p_i > 0 \}$.
\end{proposition}
\begin{proof}
The proof follows the same lines as the proof of Proposition \ref{prop:bounds_LLd} noting that as $\lambda$ gets close to zero, one will only find exclusively busy servers when probing $d_j$ servers.
\end{proof}

\section{Numerical Experiments} \label{sec:num_exp}
\begin{figure*}[t]
\begin{subfigure}{0.45\textwidth}
\centering
\captionsetup{width=.8\linewidth}
\includegraphics[width=1\linewidth]{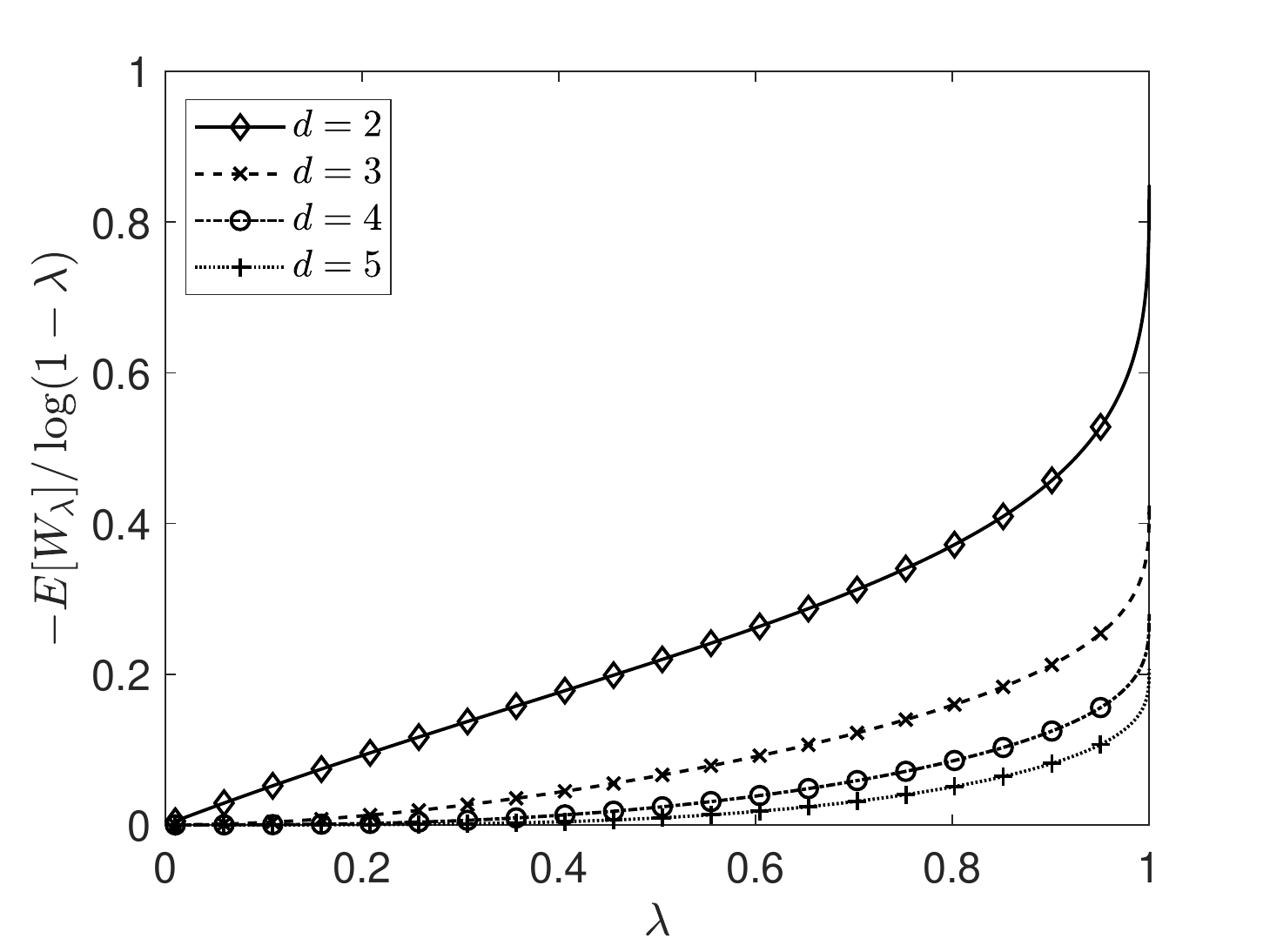}
\caption{$\E[W_\lambda]$ is scaled by $-\log(1-\lambda)$.}
\label{fig1a}
\end{subfigure}
\begin{subfigure}{0.45\textwidth}
\centering
\captionsetup{width=.8\linewidth}
\includegraphics[width=1\linewidth]{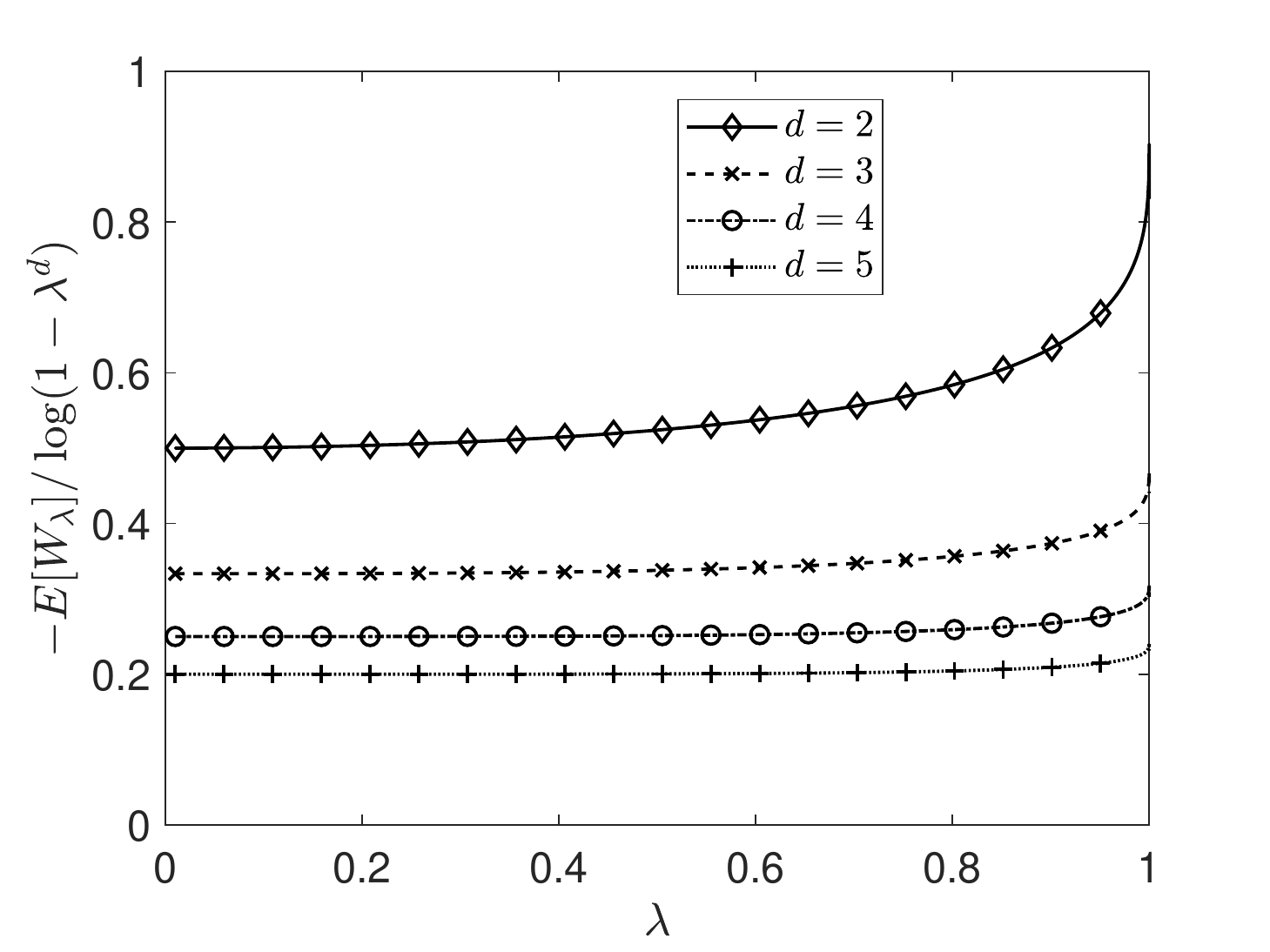}
\caption{$\E[W_\lambda]$ is scaled by $-\log(1-\lambda^d)$.}
\label{fig1b}
\end{subfigure}
\caption{$\E[W_\lambda]$ as a function of $\lambda$ for the LL($d$) policy with
different scalings.}
\label{fig1}
\end{figure*}

\begin{figure*}[t]
\begin{subfigure}{0.45\textwidth}
\centering
\captionsetup{width=.8\linewidth}
\includegraphics[width=1\linewidth]{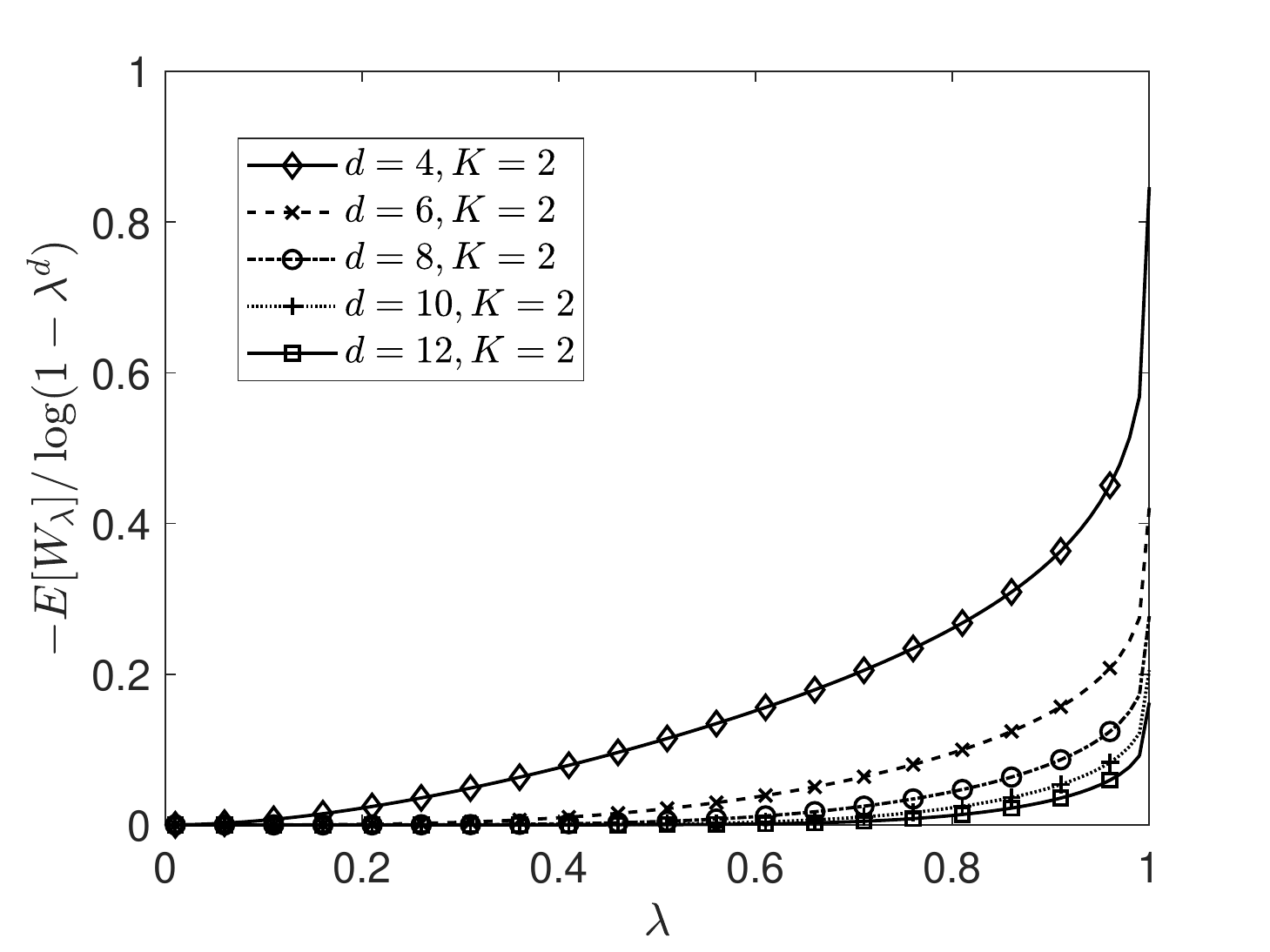}
\caption{$\E[W_\lambda]$ is scaled by $-\log(1-\lambda)$.}
\label{fig2a}
\end{subfigure}
\begin{subfigure}{0.45\textwidth}
\centering
\captionsetup{width=.8\linewidth}
\includegraphics[width=1\linewidth]{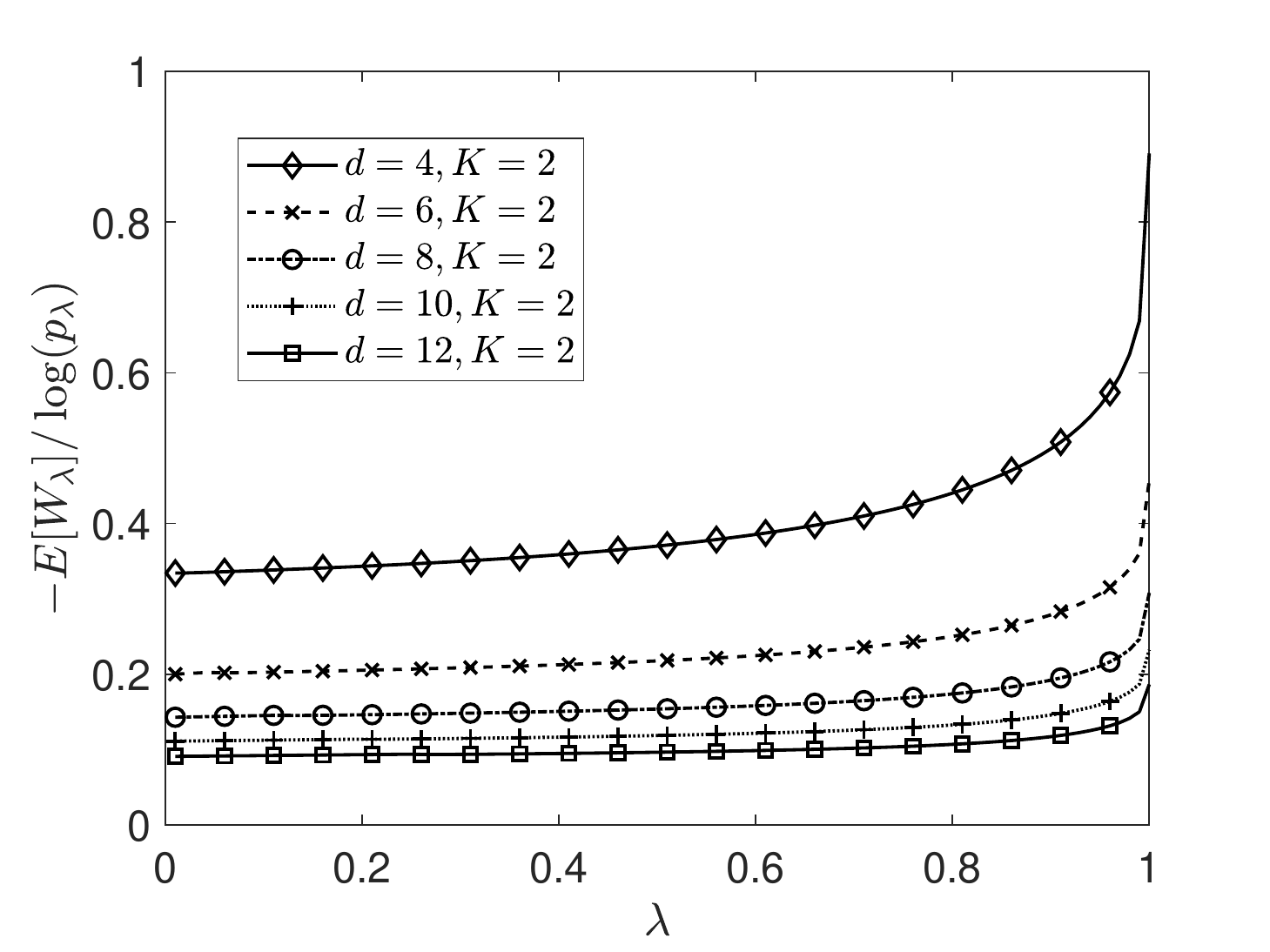}
\caption{$\E[W_\lambda]$ is scaled by $-\log(p_\lambda)$, where $p_\lambda$ is given as in Proposition \ref{prop:LLdK_lb}.}
\label{fig2b}
\end{subfigure}
\caption{$\E[W_\lambda]$ as a function of $\lambda$ for the LL($d, K$) policy with
different scalings.}
\label{fig2}
\end{figure*}
In Figure \ref{fig1a}, we show the evolution of $-\frac{\E[W_\lambda]}{\log(1-\lambda)}$ as a function of $\lambda$ for the LL($d$) policy. We observe that (as stated before) the high load limit is given as $\frac{1}{d-1}$, while the low load limit is simply zero. However in Figure \ref{fig1b}, we used the alternate scaling $-\log(1-\lambda^d)$ as argued in Section \ref{sec:low_traffic}.  Figure \ref{fig1b} also shows that the alternate
scaling flattens the curve, which shows that $-\log(1-\lambda^d)/d$ resp~$-\log(1-\lambda^d)/(d-1)$ are better approximations of $\E[W_\lambda]$ (which become exact as $\lambda \approx 0$ resp.~$\lambda \approx 1$).

In Figure \ref{fig2} we repeat the experiment conducted in Figure \ref{fig1} for the LL($d, K$) policy instead of the LL($d$) policy. We fix $K=2$ and vary the value of $d$, we again observe that using the alternate scaling $-\log(p_\lambda)$ yields a flatter curve and thus better
approximation for $\lambda$ close to zero or one.

\section{Conclusions and Extensions}\label{sec:concl}
\subsection*{Conclusions}
In this paper we studied the behaviour of the expected waiting time $\E[W_\lambda]$ when $\lambda \approx 1$ for a variety of load balancing policies in the mean field regime. We present a set of sufficient assumptions such that the limit $\lim_{\lambda \rightarrow 1^-} -\E[W_\lambda]/\log(1-\lambda)$ can be derived without much effort. For some load balancing policies (such as LL($d$)) these assumptions are easy to verify, while for other polices (such as LL($d,K$)) this turned out to be much more challenging. Even if it is unclear how to verify these assumptions, our result yields a natural conjecture on the limiting value. The resulting limiting value is also surprisingly elegant for the policies studied in this paper. Adjusting the denominator allows one to retain the limiting value as $\lambda \rightarrow 1^-$ while also obtaining an elegant low load limit.
As our main theorem applies to any ODE for which $T_\lambda$ satisfies the sufficient assumptions, our main results may also find applications outside the area of load balancing. 

\subsection*{Extensions}
In this paper, we focused on workload dependent load balancing policies. For a queue length dependent policy such as SQ($d$), one denotes by $u_k$ the probability that the cavity queue has $k$ or more jobs waiting for service. In the mean field regime these probabilities $u_k$ can be expressed via a recursive relation $u_{k+1} = T_\lambda(u_k)$, where the function $T_\lambda$ is the same as  for the workload dependent version of the same policy. For example, for SQ($d, K$) one can show that $u_{k+1} = T_\lambda(u_k)$, where $T_\lambda$ is defined as in \eqref{eq:LLdK_Tlam}. This allows one to establish a theorem, similar to Theorem \ref{thm:gen_result_ODE}, that for the queue length dependent variants the limit for $\lambda$ tending to one is given by:
$$
\lim_{\lambda \rightarrow 1^-} - \frac{\E[W_\lambda]}{\log(1-\lambda)} = \frac{B}{\log(A)}.
$$
Note that the denominator now equals $\log(A)$ instead of $A-1$. 
In particular, for the SQ($d, K$) policy this implies that $\lim_{\lambda \rightarrow 1^-} - \frac{\E[W_\lambda]}{\log(1-\lambda)} = \frac{1}{\log(\frac{d}{K})}$, which is 
a generalization of the result in \cite{mitzenmacher2001power} for $K=1$. 
In the queue length dependent case the assumptions on $T_\lambda$ are the same as for the workload dependent load balancing policies, except that $\bar w_\lambda$ in assumption \ref{item:ODE_only_new_req} is replaced by a $\bar k _\lambda \in \mathbb{N}$.

Although the paper is restricted to exponential job sizes, numerical experiments (not reported in the paper) suggest that the observations made in Figure \ref{fig0} also hold for non-exponential job size distributions, which suggests that the main ideas presented in this paper may  be more generally applicable. 


%
%

\bibliography{thesis}

\section{Appendix}
\subsection{Proof of Proposition \ref{prop:SQdK_recursion_correct}} \label{app:proof_SQdK_recursion_correct}
\begin{proof}
We may compute:
\begin{align*}
T_\lambda(u)  &= \lambda \sum_{j=0}^{K-1} \frac{K-j}{K} \binom{d}{j} u^{d-j}
(1-u)^j \leq \lambda \sum_{j=0}^{K-1} \frac{d-j}{d} \binom{d}{j} u^{d-j} (1-u)^j = \lambda \sum_{j=0}^{K-1} \frac{d-j}{d} \binom{d}{d-j} u^{d-j} (1-u)^j\\
&= \lambda \sum_{j=0}^{K-1} \binom{d-1}{d-j-1} u^{d-j} (1-u)^j = \lambda u \sum_{j=0}^{K-1} \binom{d-1}{j} u^{d-1-j} (1-u)^j,
\end{align*}
and this last sum is bounded by one, from which the result follows.
\end{proof}

\subsection{Proof of Lemma \ref{lem:Tuoveru_increasing_SQdK}} \label{proof_lem:Tuoveru_increasing_SQdK}
\begin{proof}
We may divide both sides in \eqref{eq:Tuoveru_increasing} by $u^{d-2}$, we compute:
\begin{align*}
\frac{1}{u^{d-2}}\left( \frac{T_\lambda(u)}{u} \right)'
&= \frac{\lambda}{K} \sum_{j=0}^{K-1} (K-j) \binom{d}{j} (d-j-1) \left( \frac{1-u}{u} \right)^j 
- \frac{\lambda}{K} \sum_{j=1}^{K-1} (K-j) \binom{d}{j} j \left( \frac{1-u}{u} \right) ^{j-1}.
\end{align*}
Now let $\xi = \frac{1-u}{u}$ and note that $\xi \in (0,\infty)$ for $u \in (0,1)$, we find that $\frac{1}{u^{d-2}}\left( \frac{T_\lambda(u)}{u} \right)'$ can be further simplified as:
\begin{align*}
&\frac{\lambda}{K} \binom{d}{K-1} (d-K) \xi^{K-1} + \frac{\lambda}{K} \sum_{j=0}^{K-2} \bigg[ (K-j) \binom{d}{j} (d-j-1) - (K-j-1) \binom{d}{j+1} (j+1)\bigg] \xi^j \\
&= (d-K) \frac{\lambda}{K} \sum_{j=0}^{K-1} \binom{d}{j} \xi^j > 0.
\end{align*}
This shows that \eqref{eq:Tuoveru_increasing} holds.
\end{proof}

\subsection{Proof of Lemma \ref{lem:fplusag}} \label{app:lem:fplusag}
\begin{proof}
As $[0,1]$ is compact and $g$ is continuous we have $A= \max_{x\in [0,1]} \{ g(x) \}<\infty$. Moreover as $f$ is continuous and decreasing in $0$, we find a $\delta, \gamma > 0$ such that $f(\delta)=-\gamma$ and $f'(x)<0$ for all $x\in [0,\delta]$. If we now let $a_0 = \frac{\gamma}{A}$, the result easily follows
as $h_a(0) \geq 0$ and $h_a(\delta) \leq 0$.
\end{proof}

\subsection{Proof of Lemma \ref{lem:SQdk1}} \label{app:lem:SQdk1}
\begin{proof}
Dividing both sides of $u-T_\lambda(u)=0$ by $\lambda \cdot u^d$ we find this equation to be equivalent to:
$$
\frac{1}{\lambda}\frac{1}{u^{d-1}} - \frac{1}{K} \sum_{j=0}^{K-1} (K-j) \binom{d}{j} \left( \frac{1-u}{u} \right)^j = 0.
$$
Let $\xi = \frac{u-1}{u} \in [0,1]$ for $u \geq 1$, we obtain:
$$
\frac{(1-\xi)^{d-1}}{\lambda} - \frac{1}{K} \sum_{j=0}^{K-1} (K-j) \binom{d}{j} (-1)^j \xi^j
=0
$$
adding and subtracting $(1-\xi)^{d-1}$, we further find this to be equivalent to
$$
\left( \frac{1}{\lambda} - 1 \right) (1-\xi)^{d-1} + (1-\xi)^{d-1} - \frac{1}{K} \sum_{j=0}^{K-1} (K-j) \binom{d}{j} (-1)^j \xi^j = 0.
$$
If we let $a=\frac{1}{\lambda} - 1$ we find a value $a_0>0$ from Lemma \ref{lem:fplusag} such that there exists a root at $\xi_a$ for all $a \in [0,a_0)$ (as $f'(0)=1-d/K < 0$).
 It now suffices to take $\bar \lambda \geq \frac{1}{1+a_0}$ from which the existence of a root for $T_\lambda(u)=u$ follows, moreover \ref{enum:SQdK1} trivially follows from Lemma \ref{lem:fplusag} as we define $u_\lambda=\min\{u \in [1,\infty) \mid T_{\lambda}(u_\lambda)=u_\lambda \}$.

To show \ref{enum:ulamdK8} we note that:
\begin{align*}
\lim_{\lambda \rightarrow 1^-} h_{\lambda}(\varepsilon)
&= \left( \frac{1-(1-\varepsilon)^d}{\varepsilon} \right) - \frac{1}{K} \sum_{j=1}^{K-1} (K-j) \binom{d}{j} (1-\varepsilon)^{d-j} \varepsilon^{j-1}.
\end{align*}
Taking the limit $\varepsilon \rightarrow 0^+$ of this expression yields the sought result.
\end{proof}

\subsection{Proof of Lemma \ref{lem:Dulamn}} \label{app:proof:lem:Dulamn}
\begin{proof}
We first show that for $\Theta(u)=\frac{1}{\lambda}\frac{T_\lambda(u)}{u}$:
\begin{align}
\Theta^{(n)}(1) &= (-1)^{n+1} n! \frac{d-K}{K},  \label{eq:Dtheta_lowerK}
\end{align}
for $1 \leq n \leq K $  and
\begin{align}
\Theta^{(K+1)}(1) &= (-1)^{K+1} \frac{d! - (d-K)! (K+1)!}{(d-K)!} \frac{d-K}{K}. \label{eq:Dtheta_Kplus1}
\end{align}
We showed in the proof of Lemma \ref{lem:Tuoveru_increasing_SQdK} that:
$$
\frac{K}{d-K} \Theta'(u) = \sum_{j=0}^{K-1} \binom{d}{j} u^{d-j-2} (1-u)^j.
$$
By induction on $n$ we now show for $n \leq K$ that
\begin{align}\label{eq:Theta_n_diff}
\frac{K}{d-K} \Theta^{(n)}(u)&= (-1)^{n+1} n! \sum_{j=0}^{K-n} \binom{d}{j} u^{d-j-n-1} (1-u)^j + \sum_{j=1}^{n-1} E_j^{(n-j-1)},
\end{align}
where we denote:
\begin{align}\label{eq:Ej}
E_j= (-1)^{j+1} j! \binom{d}{K-j} (d-K-1) u^{d-K-2} (1-u)^{K-j}.
\end{align}
Indeed, one finds:
\begin{align*}
\frac{K}{d-K} &\Theta^{(n)}(u)
= \frac{K}{d-K} \frac{\partial}{\partial u} \bigg[ (-1)^n (n-1)! \sum_{j=0}^{K-n+1} \binom{d}{j} u^{d-j-n} (1-u)^j + \sum_{j=1}^{n-2} E_j^{(n-j-2)} \bigg]\\
&= \frac{K}{d-K} \bigg[ (-1)^{n+1} n! \sum_{j=0}^{K-n} \frac{\binom{d}{j+1}(j+1) - \binom{d}{j} (d-j-n)}{n} u^{d-j-n-1} (1-u)^j\\
& + (-1)^n (n-1)! \binom{d}{K-n+1} (d-K-1) u^{d-K-2} (1-u)^{K-n+1} +\sum_{j=1}^{n-2} E_j^{(n-j-1)}\bigg].
\end{align*}
The result then follows by induction by applying the equality:
$$
\frac{\binom{d}{j+1}(j+1) - \binom{d}{j} (d-j-n)}{n} = \binom{d}{j}.
$$
Noting that for any $n \leq K$ we have $\lim_{u\rightarrow 1} \sum_{j=1}^{n-2}  E_j^{(n-j-1)} = 0$, we find that \eqref{eq:Dtheta_lowerK} indeed holds. Furthermore we have:
$$
\frac{K}{d-K} \Theta^{(K+1)}(u)= (-1)^{K+1} K! (d-K-1) u^{d-K-2} + \sum_{j=1}^{K-1} E_j^{(K-j)}.
$$
Moreover, it is not hard to see that:
$$
\lim_{u \rightarrow 1} E_j^{(K-j)} = (-1)^{K+1} \binom{d}{K-j} (d-K-1) j! (K-j)!.
$$
This allows us to compute:
\begin{align*}
\frac{K}{d-K}& \Theta^{(K+1)} (1) = (-1)^{K+1} \left[ K! + \sum_{j=1}^{K-1} \binom{d}{K-j} j! (K-j)! \right] (d-K-1)\\
&= (-1)^{K+1}K! \left[ \sum_{j=0}^{K-1} \frac{\binom{d}{j}}{\binom{K}{j}} \right]  (d-K-1)
= (-1)^{K+1}K! \left(\binom{d}{K}-(K+1)\right),
\end{align*}
where we used identity (4.1) in \cite[p46]{gould1972combinatorial} with $j=0, n=K-1, z=d$ and $x=K$. This shows that \eqref{eq:Dtheta_Kplus1} indeed holds.

We now continue by induction to show (\ref{eq:Dulam_lowerK}-\ref{eq:Dulam_Kplus1}). For the case $n=1$ we note that from $u_{\lambda}=T_\lambda(u_{\lambda})$ it follows that:
\begin{align}
u_{\lambda}' &= 
\frac{1}{K} \sum_{j=0}^{K-1} (K-j) \binom{d}{j} u_\lambda ^{d-j} (1-u_\lambda)^j + \frac{\lambda}{K} \sum_{j=0}^{K-1} (K-j) (d-j) \binom{d}{j} u_\lambda^{d-j-1} (1-u_\lambda)^j u_\lambda' \nonumber\\
& -\frac{\lambda}{K} \sum_{j=1}^{K-1} (K-j) j \binom{d}{j} u_\lambda^{d-j} (1-u_\lambda)^{j-1} u_\lambda'. \label{eq:Dulam_LLdk}
\end{align}
Taking the limit $\lambda \rightarrow  1^-$ we obtain $\lim_{\lambda \rightarrow 1^-} u_{\lambda}' = 1 + \frac{d}{K} \lim_{\lambda \rightarrow 1^-} u_{\lambda}'$ yielding \eqref{eq:Dulam_lowerK} with $n=1$. Let $2 \leq n \leq K+1$ note that $u_\lambda= T_\lambda(u_\lambda)$ and therefore also $1=\lambda \Theta(u_\lambda)$. By differentiating both sides $n\geq 2$ times, it follows that we have:
\begin{align}\label{eq:1=lamTheta}
0= n \left( \frac{\partial}{\partial \lambda} \right)^{n-1} \Theta(u_\lambda) + \lambda \left( \frac{\partial}{\partial \lambda} \right)^{n} \Theta(u_\lambda).
\end{align}
It follows from the Fa\`a di Bruno formula that:
\begin{equation}\label{eq:bruno}
\left(\frac{\partial}{\partial \lambda}\right)^n \Theta (u_\lambda)
=
\sum_{k=1}^n \Theta^{(k)}(u_\lambda) B_{n,k}(u_\lambda',\dots, u_{\lambda}^{(n-k+1)})
\end{equation}
where $B_{n,k}$ denotes the exponential Bell polynomial. We have:
$$
B_{n,1}(u_\lambda',\dots,u_\lambda^{(n)})=u_\lambda^{(n)}
$$
and for $k >1$ induction allows us to state for $n \leq K+1$:
\begin{align*}
\lim_{\lambda \rightarrow 1^-} B_{n,k}(u_\lambda',\dots,u_{\lambda}^{(n-k+1)})&=
B_{n,k}\big( (-1) \frac{K}{d-K}, \ldots, (-1)^{n-k+1} (n-k+1)! \frac{d^{n-k} K}{(d-K)^{n-k+1}} \big)\\
&= B_{n,k}(1!,\dots,(n-k+1)!) \cdot (-1)^n \frac{d^{n-k} K^k}{(d-K)^n},
\end{align*}
where we used the simple identities
\begin{align*}
B_{n,k}(x_1y, x_2y, \ldots,x_{n-k+1}y) = B_{n,k}(x_1,\ldots,x_{n-k+1}) y^k, \\
B_{n,k}(x_1z, x_2z^2, \ldots,x_{n-k+1}z^{n-k+1}) = B_{n,k}(x_1,\ldots,x_{n-k+1}) z^n,
\end{align*}
with $y=K/d$ and $z=-d/(d-K)$.
Using \eqref{eq:Lah} we have:
$$
\lim_{\lambda \rightarrow 1^-} B_{n,k}(u_\lambda',\dots,u_{\lambda}^{(n-k+1)})
=
\frac{n!}{k!} \binom{n-1}{k-1} (-1)^n \frac{d^{n-k} K^k}{(d-K)^n}.
$$
Analogously, one may compute:
$$
\lim_{\lambda \rightarrow 1^-} B_{n-1,k}(u_\lambda',\dots,u_{\lambda}^{(n-k)})
=
\frac{(n-1)!}{k!} \binom{n-2}{k-1} (-1)^{n-1} \frac{d^{n-k-1} K^k}{(d-K)^{n-1}}.
$$
Therefore \eqref{eq:1=lamTheta}, \eqref{eq:bruno} and \eqref{eq:Dtheta_lowerK} imply for $n \leq K+1$:
\begin{align*}
0
&=n \sum_{k=1}^{n-1} \frac{(n-1)!}{k!} \binom{n-2}{k-1} (-1)^{n-1} \frac{d^{n-k-1} K^k}{(d-K)^{n-1}} \Theta^{(k)}(1)\\
& +\sum_{k=2}^n \frac{n!}{k!} \binom{n-1}{k-1} (-1)^n \frac{d^{n-k} K^k}{(d-K)^n} \Theta^{(k)}(1) +\Theta^{(1)}(1) \lim_{\lambda \rightarrow 1^-} u_{\lambda}^{(n)}\\
&= n! \sum_{k=1}^{n-1} \binom{n-2}{k-1} (-1)^{n+k} \frac{d^{n-k} K^{k-1}}{(d-K)^{n-1}}
+ n! \sum_{k=1}^{n-1} \binom{n-2}{k-1} (-1)^{n+k+1}\frac{d^{n-k-1} K^{k}}{(d-K)^{n-1}}\\
&+ n! \sum_{k=2}^{n-1} \binom{n-1}{k-1} (-1)^{n+k+1} \frac{d^{n-k} K^{k-1}}{(d-K)^{n-1}}\\
&+ (-1)^n \left(\frac{K}{d-K}\right)^n \Theta^{(n)}(1) + \frac{d-K}{K} \lim_{\lambda \rightarrow 1^-} u_\lambda^{(n)}\\
&=n! (-1)^{n+1} \left( \frac{d}{d-K} \right)^{n-1} + n! \left( \frac{K}{d-K} \right)^{n-1}
+ n! \sum_{k=2}^{n-1} \left[ \binom{n-2}{k-1} + \binom{n-2}{k-2} - \binom{n-1}{k-1} \right] \cdot \\
&(-1)^{n+k} \frac{d^{n-k} K^{k-1}}{(d-K)^{n-1}}+(-1)^n \left( \frac{K}{d-K} \right)^n \Theta^{(n)}(1) + \frac{d-K}{K} \lim_{\lambda \rightarrow 1^-} u_{\lambda}^{(n)}.
\end{align*}
From Pascal's triangle we find that:
\begin{align*}
\lim_{\lambda\rightarrow 1^-} u_\lambda^{(n)}
&=
(-1)^{n+1} \left( \frac{K}{d-K} \right)^{n+1} \Theta^{(n)}(1) -n! \left( \frac{K}{d-K} \right)^n + n! (-1)^n \frac{d^{n-1} K}{(d-K)^n},
\end{align*}
for $n \leq K+1$.
Plugging in 
 \eqref{eq:Dtheta_lowerK} and \eqref{eq:Dtheta_Kplus1} yields
  \eqref{eq:Dulam_lowerK} and \eqref{eq:Dulam_Kplus1}, respectively.
\end{proof}

\subsection{Proof of Lemma \ref{lem:LLdk2}} \label{app:proof:lem:LLdk2}
\begin{proof}
Throughout, we assume that $\bar \lambda < \lambda < 1$ with $\bar \lambda$ as in Lemma \ref{lem:SQdk1}. We show there is some $\tilde{\lambda} \geq \bar \lambda$ and $b \in \mathbb{N}$ which does not depend on the value of $\lambda$ for which $h_{\lambda}(x)$ is decreasing on $[u_\lambda - \lambda^b, u_{\lambda}]$ for all $\lambda \in (\tilde \lambda, 1)$.

First we note that the derivative of $T_\lambda(u)$ as a function of $u$ is:
\begin{align*}
T_\lambda'(u)
&=\frac{\lambda}{K} \sum_{j=0}^{K-1} (K-j) (d-j) \binom{d}{j} u^{d-j-1} (1-u)^j 
-\frac{\lambda}{K} \sum_{j=0}^{K-2} (K-j-1) (j+1) \binom{d}{j+1} u^{d-j-1} (1-u)^j\\
&= \frac{\lambda}{K} \sum_{j=0}^{K-1} (d-j) \binom{d}{j} u^{d-j-1} (1-u)^j.
\end{align*}
We thus find:
\begin{align*}
Kx^2 h_\lambda'(x)
&= \bigg[ -\frac{1}{x^2} (u_\lambda - T_\lambda(u_\lambda - x)) + \frac{1}{x} T_\lambda'(u_\lambda - x) \bigg] \cdot K x^2\\
&=-Ku_\lambda + \lambda \sum_{j=0}^{K-1} (K-j) \binom{d}{j} (u_\lambda-x)^{d-j} (1-u_\lambda+x)^j\\
&+\lambda x \sum_{j=0}^{K-1} (d-j) \binom{d}{j} (u_\lambda - x)^{d-j-1} (1-u_\lambda + x)^j\\
&=-K u_\lambda + \lambda\sum_{j=0}^{K-1} (K-j) \binom{d}{j} (u_\lambda - x)^{d-j} (1-u_\lambda+x)^j\\
&-\lambda\sum_{j=0}^{K-1} (d-j) \binom{d}{j} (u_\lambda - x)^{d-j} (1-u_\lambda + x)^j\\
&+\lambda u_\lambda \sum_{j=0}^{K-1} (d-j) \binom{d}{j} (u_\lambda - x)^{d-j-1} (1-u_\lambda+x)^j.
\end{align*}
If we now define $\zeta_\lambda(x)=Kx^2h_\lambda'(x)$ we obtain:
\begin{align}
\zeta_\lambda(x)
&=
-K u_\lambda - \lambda (d-K) \sum_{j=0}^{K-1} \binom{d}{j} (u_\lambda - x)^{d-j} (1-u_\lambda+x)^j\nonumber\\
&+ \lambda u_\lambda \sum_{j=0}^{K-1} (d-j) \binom{d}{j} (u_\lambda-x)^{d-j-1} (1-u_\lambda+x)^j.\label{eq:zetalamx}
\end{align}
It therefore suffices to show that $\zeta_\lambda(x) \leq 0$ for $\lambda$ sufficiently close to one. To this end we compute:
\begin{align*}
\zeta_\lambda'(x)
&= \lambda (d-K) \binom{d}{K-1} (d-K+1) (u_\lambda - x)^{d-K} (1-u_\lambda+x)^{K-1}\\
&+ \lambda (d-K) \sum_{j=0}^{K-2} \binom{d}{j} (d-j) (u_\lambda - x)^{d-j-1} (1-u_\lambda + x)^j\\
&-\lambda (d-K) \sum_{j=0}^{K-2} \binom{d}{j+1} (j+1) (u_\lambda - x)^{d-j-1} (1-u_\lambda + x)^{j}\\
&- \lambda u_\lambda (d-K) (d-K+1) \binom{d}{K-1} (u_\lambda - x)^{d-K-1} (1-u_\lambda+x)^{K-1}\\
&-\lambda u_\lambda \sum_{j=0}^{K-2} (d-j) (d-j-1) \binom{d}{j} (u_\lambda -x)^{d-j-2} (1-u_\lambda+x)^j\\
&+ \lambda u_{\lambda} \sum_{j=0}^{K-2} \binom{d}{j+1} (d-j-1) (j+1) (u_\lambda - x)^{d-j-2} (1-u_\lambda +x)^j,
\end{align*}
which simplifies to:
\begin{equation} \label{eq:DzetaLam}
\zeta_\lambda'(x)=-\lambda (d-K) \binom{d}{K-1} (d-K+1) (1-u_\lambda+x)^{K-1} (u_\lambda-x)^{d-K-1} x.
\end{equation}
This is obviously negative for all $x \in [u_\lambda-1,u_\lambda]$. It thus suffices to show that we can find a value $b \in \mathbb{N}$ such that $\zeta_\lambda(u_\lambda-\lambda^b) \leq 0$. To this end, we find:
\begin{align*}
\zeta_\lambda(u_\lambda-\lambda^b)
&= -K u_\lambda - \lambda(d-K) \sum_{j=0}^{K-1} \binom{d}{j} \lambda^{b(d-j)} (1-\lambda^b)^j
+\lambda u_\lambda \sum_{j=0}^{K-1} (d-j)  \binom{d}{j} \lambda^{b(d-j-1)} (1-\lambda^b)^j\\
&= -K u_\lambda - \lambda (d-K) \sum_{j=0}^{K-1} \binom{d}{j} \lambda^{b(d-j)} (1-\lambda^b)^j
+ \frac{\lambda u_\lambda}{\lambda^b} (d-K) \sum_{j=0}^{K-1} \binom{d}{j} \lambda^{b(d-j)} (1-\lambda^b)^j\\
&+\frac{\lambda u_{\lambda}}{\lambda^b} \sum_{j=0}^{K-1} (K-j) \binom{d}{j} \lambda^{b(d-j)} (1-\lambda^b)^j\\
&= -\frac{K\lambda}{\lambda^b} \big(\lambda^b \frac{T_\lambda(u_\lambda)}{\lambda} - u_\lambda \frac{T_\lambda(\lambda^b)}{\lambda} \big)
+ \frac{\lambda}{\lambda^b} (d-K) (u_\lambda - \lambda^b) \sum_{j=0}^{K-1} \binom{d}{j} \lambda^{b(d-j)} (1-\lambda^b)^j.
\end{align*}
Now let us denote $\Theta(u) = \frac{1}{\lambda} \frac{T_\lambda(u)}{u}$, we find that:
\begin{align}
\lambda^b \zeta_\lambda\left(u_\lambda-\lambda^b\right)
&= -\lambda K u_\lambda \lambda^b (\Theta(u_\lambda) - \Theta(\lambda^b))
 + \lambda (d-K) (u_\lambda - \lambda^b) \sum_{j=0}^{K-1} \binom{d}{j} \lambda^{b(d-j)} (1-\lambda^b)^j. \label{eq:lambzetaulamminlamb}
\end{align}
For now let us focus on the case $K=d-1$. By \eqref{eq:Theta_n_diff} and \eqref{eq:Ej} we have 
for $n \leq K = d-1$ that
\begin{equation}\label{eq:DnTheta}
\Theta^{(n)}(u)=(-1)^{n+1} n! \frac{1}{d-1} \sum_{j=0}^{d-n-1} \binom{d}{j} u^{d-j-n-1} (1-u)^j,
\end{equation}
as $E_j=0$ when $K=d-1$. Note that $\Theta^{(d-1)}(u)=(-1)^d (d-2)!$ is constant and therefore
$\Theta^{(n)}(u)=0$ for $n > K=d-1$. 

Employing the Taylor expansion of $\Theta(u_\lambda)$ at $\lambda^b$, we find that \eqref{eq:lambzetaulamminlamb} can be written as:
\begin{align}
\lambda^b \zeta_\lambda(u_\lambda-\lambda^b)
&=
\lambda \bigg[ (u_\lambda - \lambda^b) \sum_{j=0}^{d-2} \binom{d}{j} \lambda^{b(d-j)} (1-\lambda^b)^j
- u_\lambda \lambda^b \sum_{n=1}^{d-1} (d-1) \Theta^{(n)}(\lambda^b) \frac{(u_\lambda-\lambda^b)^n}{n!}\bigg].\label{eq:taylor_exp}
\end{align}
Due to \eqref{eq:DnTheta}
\begin{align*}
u_\lambda \lambda^b \sum_{n=1}^{d-1} (d-1) \Theta^{(n)}(\lambda^b) \frac{(u_\lambda-\lambda^b)^n}{n!} 
&= u_\lambda \sum_{n=1}^{d-1} (-1)^{n+1} (u_\lambda-\lambda^b)^n \sum_{j=0}^{d-n-1} \binom{d}{j} \lambda^{b(d-j-n)} (1-\lambda^b)^j\\
&=-u_\lambda \sum_{j=0}^{d-2} \binom{d}{j} \lambda^{b(d-j)} (1-\lambda^b)^j \sum_{n=1}^{d-j-1} 
 (1-u_\lambda/\lambda^b)^n\\
 &=-u_\lambda \sum_{j=0}^{d-2} \binom{d}{j} \lambda^{b(d-j)} (1-\lambda^b)^j \left(
\frac{1-(1-u_\lambda/\lambda^b)^{d-j}}{u_\lambda/\lambda^b}-1 \right)\\
 &= \sum_{j=0}^{d-2} \binom{d}{j} (1-\lambda^b)^j \left(
\lambda^{b(d-j)}  (u_\lambda - \lambda^b)+\lambda^b (\lambda^b-u_\lambda)^{d-j}\right).   
\end{align*}

Combined with \eqref{eq:taylor_exp} this yields:
\begin{align*}
\lambda^b \zeta_\lambda(u_\lambda-\lambda^b)
&= \lambda^{b+1} \sum_{j=0}^{d-2} (-1)^{d-j+1} \binom{d}{j} (u_\lambda-\lambda^b)^{d-j} (1-\lambda^b)^j.
\end{align*}
Dividing by $(u_\lambda-\lambda^b)^d$ we find that:
\begin{align*}
\frac{\lambda^b}{(u_\lambda-\lambda^b)^d} \zeta_\lambda(u_\lambda-\lambda^b)
&=\lambda^{b+1} \sum_{j=0}^{d-2} (-1)^{d-j+1} \binom{d}{j} \left(\frac{1-\lambda^b}{u_\lambda-\lambda^b}\right)^j.
\end{align*}
It is easy to show by applying l'Hopital's rule and using the fact that $\lim_{\lambda \rightarrow 1^-}  u_\lambda' = -d+1$ for $K=d-1$ that: 
$$
\lim_{\lambda \rightarrow 1^-} \left( \frac{1-\lambda^b}{u_\lambda - \lambda^b} \right)
=
\frac{b}{b+d-1}.
$$
Therefore we find
\begin{align*}
&\lim_{\lambda \rightarrow 1^-} \frac{\lambda^b}{(u_\lambda-\lambda^b)^d} \zeta_\lambda(u_\lambda-\lambda^b) = \sum_{j=0}^{d-2} (-1)^{d-j+1} \binom{d}{j} \left(\frac{b}{b+d-1}\right)^j\\
&=-d \left( \frac{b}{b+d-1} \right)^{d-1} + \left(\frac{b}{b+d-1}\right)^d + (-1)^{1+d} \left( \frac{d-1}{b+d-1} \right)^d,
\end{align*}
which converges to $1-d \leq 0$ as $b$ tends to infinity. This proves Lemma \ref{lem:LLdk2}
for $K=d-1$.

Fix $K$ and let $d \geq K+1$ be variable, we find (apply (\ref{eq:zetalamx}-\ref{eq:DzetaLam}) that:
\begin{align*}
\zeta_\lambda'(x)= - \lambda(d-K) \binom{d}{K-1} (d-K+1) (1-u_\lambda+x)^{K-1}(u_\lambda-x)^{d-K-1} x,
\end{align*}
and
\begin{align}\label{eq:zeta(ulam-1)}
\zeta_\lambda(u_\lambda-1)=(\lambda d - K) u_\lambda - \lambda (d-K).&
\end{align}
Now let $(K_1,d_1)$ and $(K_2,d_2)$ be arbitrary (with $K_i < d_i$), denote by $_i u _\lambda$ the fixed point associated to $(K_i,d_i)$ and $_i\zeta_\lambda$ the associated $\zeta_\lambda$ function. We show the following inequalities :
\begin{align}\label{eq:ineq_zeta_prime}
_2 \zeta _\lambda'(x) & \leq\  _1\zeta _\lambda'(x+\ _1u_\lambda -\ _2u_\lambda) 
\end{align}
for  $x \in [\ _2u_\lambda - 1,\ _2u_\lambda-\lambda^b]$ and
\begin{align}
_2\zeta_\lambda(_2 u _\lambda - 1) &\leq\ _1\zeta _\lambda(_1 u _\lambda - 1), & \label{eq:ineq_zeta_ulam_min_1}
\end{align}
in case we have:
\begin{enumerate}[label=(\roman*),topsep=0pt]
\item \label{enum:Keven} $K$ is even, $(K_1,d_1)=(K,d)$ and $(K_2,d_2)=(K,d+1)$,
\item \label{enum:Kodd} $K$ is odd, $(K_1,d_1)=(K+1,d+1)$ and $(K_2,d_2)=(K,d)$.
\end{enumerate}
If (\ref{eq:ineq_zeta_prime}-\ref{eq:ineq_zeta_ulam_min_1}) hold, we find that:
\begin{align*}
_2\zeta_\lambda(_2u_\lambda-\lambda^b)
&=\ _2\zeta_\lambda( _2u_\lambda-1)+\int_{_2 u_\lambda-1}^{_2 u_\lambda-\lambda^b} \ _2 \zeta_\lambda'(x)\, dx\\
&\leq\ _1 \zeta_\lambda(_1u_\lambda-1)+\int_{_2u_\lambda-1}^{_2u_\lambda-\lambda^b} \ _1 \zeta_\lambda'(x+\ _1u_{\lambda} -\ _2u_{\lambda})\, dx
=\ _1 \zeta_\lambda(_1 u_{\lambda} - \lambda^b)
\end{align*}
This shows that if $\ _1 \zeta_\lambda(\ _1 u_{\lambda} - \lambda^b) \leq 0$, then also $_2 \zeta_\lambda(_2u_\lambda-\lambda^b) \leq 0$. Applying \ref{enum:Keven} would then conclude the proof for $K$ even as we already established the result for $K=d-1$. Having shown the result for $K$
even then implies that the result also holds for $K$ odd by applying \ref{enum:Kodd}. 
\newline
First, we show \eqref{eq:ineq_zeta_prime} for \ref{enum:Keven}. To this end we let $x \in [\ _2u_{\lambda}-1,\ _2u_{\lambda} - \lambda^d]$ be arbitrary, we find that $_2 \zeta_{\lambda}'(x) \leq\ _1 \zeta_\lambda'(x+\ _1u_{\lambda} -\ _2 u_\lambda)$ is equivalent to:
$$
\left(1 + \frac{_1 u_{\lambda} - _2 u _{\lambda}}{x}\right) \frac{1}{ _2u_\lambda -x} \leq \frac{(d_2-K) \binom{d_2}{K-1} (d_2-K+1)}{(d_1-K) \binom{d_1}{K-1} (d_1-K+1)}.
$$
This can be shown to hold for $\lambda$ sufficiently close to $1$ by noting that
for  $x \in [\ _2u_\lambda - 1,\ _2u_\lambda-\lambda^b]$ we have
\begin{align*}
\lim_{\lambda\rightarrow 1^-} &\left(1 + \frac{_1 u_{\lambda} - _2 u _{\lambda}}{x}\right) \frac{1}{ _2u_\lambda -x} 
\leq \lim_{\lambda\rightarrow 1^-} \left(1+ \frac{_1 u_{\lambda} -\ _2 u _{\lambda}}{_2u_\lambda-1}
\right) \frac{1}{\lambda^b}= \frac{d_2-K}{d_1-K}
\leq \frac{(d_2-K) \binom{d_2}{K-1} (d_2-K+1)}{(d_1-K) \binom{d_1}{K-1} (d_1-K+1)},
\end{align*}
from this we find that \eqref{eq:ineq_zeta_prime} indeed holds in case \ref{enum:Keven} for any $K$ and
thus certainly for $K$ even. 
\newline
We now consider \eqref{eq:ineq_zeta_ulam_min_1} for case \ref{enum:Keven}. Due to
\eqref{eq:zeta(ulam-1)} one finds for any $n\geq 1$ that:
\begin{equation}\label{eq:Dn_lam_zeta}
\left( \frac{\partial}{\partial \lambda} \right)^n \zeta_\lambda(u_\lambda-1)
=
n d u_\lambda^{(n-1)} + (\lambda d- K) u_\lambda^{(n)} - \delta_{\{n=1\}} (d-K).
\end{equation}
We employ \eqref{eq:Dulam_lowerK}, to conclude that for $n \leq K$:
\begin{equation}\label{eq:Dn_zeta_ulam_min_1}
\lim_{\lambda \rightarrow 1^-} \left( \frac{\partial}{\partial \lambda} \right)^n \zeta_\lambda(u_\lambda-1)=0,
\end{equation}
while for $n=K+1$ we find from (\ref{eq:Dulam_lowerK}-\ref{eq:Dulam_Kplus1}):
\begin{equation}\label{eq:DKplus1_zeta_ulam_min_1}
\lim_{\lambda \rightarrow 1^-} \left( \frac{\partial}{\partial \lambda} \right)^{K+1} \zeta_\lambda(u_\lambda-1)
=
-\frac{K \cdot d!}{(d-K)!} \cdot \left( \frac{K}{d-K} \right)^K.
\end{equation}
We now denote 
\begin{equation}\label{eq:Hn_proof_LLdK}
H_n=\lim_{\lambda \rightarrow 1^-} \left( \frac{\partial}{\partial \lambda} \right)^n \left(_2\zeta_\lambda(\ _2u_\lambda-1)-\ _1\zeta_\lambda(\ _1u_\lambda-1)\right),
\end{equation}
From \eqref{eq:Dn_zeta_ulam_min_1} we clearly have $H_n=0$ for $0 \leq n \leq K$. For $n=K+1$ we find:
$$
H_{K+1}
=
\frac{d!}{(d-K)!} \cdot \left( K \left( \frac{K}{d-K} \right)^K - (d+1) \left( \frac{K}{d+1-K} \right)^{K+1} \right).
$$
which is positive if and only if:
$$
(1+d-K)^{K+1} - (d+1) (d-K)^K > 0
$$
Letting $d=K+y$ (for $y \geq 1$) we find that this is equivalent to:
\begin{align*}
(1+y)^{K+1} - (1+K+y) y^K > 0.
\end{align*}
As $(1+y)^{K+1} - (1+K+y) y^K = \sum_{j=0}^{K-1} \binom{K+1}{j} y^j$, which is positive
for $K\geq 2$ and $y\geq 0$, we conclude that
$H_{K+1}$ is positive. 

By looking at the Taylor series expansion of $_2\zeta_\lambda(_2 u _\lambda - 1) -\ _1\zeta _\lambda(_1 u _\lambda - 1)$ in $\lambda=1$ and noting that $H_0 = \ldots = H_K = 0$, 
we note that for $\lambda$ sufficiently close to one:
$$_2\zeta_\lambda(_2 u _\lambda - 1) -\ _1\zeta _\lambda(_1 u _\lambda - 1) \approx
H_{K+1} (\lambda - 1)^{K+1}/(K+1)!,
$$
which is negative for $K$ even. This shows that \eqref{eq:ineq_zeta_ulam_min_1} indeed holds for case \ref{enum:Keven}.

We now consider \eqref{eq:ineq_zeta_prime} for \ref{enum:Kodd}, using simple computations we find that this is equivalent to
$$
\binom{d+1}{K} (1-\ _2u_\lambda + x) (x+\ _1 u_\lambda -\ _2 u _\lambda) \leq \binom{d}{K-1} x,
$$
for $x \in [\ _2u_\lambda-1, \ _2u_{\lambda}-\lambda^b]$.
It therefore suffices to show that for $\lambda$ close to one, we have:
$$
(1-\ _2u_{\lambda}+x) \left(1 + \frac{\ _1 u _\lambda -\ _2 u_\lambda}{_2 u_\lambda - 1} \right) \leq \frac{K}{d+1}.
$$
This holds as  $(1-\ _2u_{\lambda}+x)$ converges to zero as 
$\lambda \rightarrow 1^-$ and $\lim_{\lambda \rightarrow 1^-} 1 + (\  _1 u _\lambda -\ _2 u_\lambda)/(\ _2 u_\lambda - 1)  = (K+1)/K$.
\newline
The final step is to show \eqref{eq:ineq_zeta_ulam_min_1} for \ref{enum:Kodd}. If we define $H_n$ as in \eqref{eq:Hn_proof_LLdK} and make use of \eqref{eq:Dn_zeta_ulam_min_1} and \eqref{eq:DKplus1_zeta_ulam_min_1}, we find that $H_n=0$ for $n \leq K$ while for $n=K+1$ we have:
$$
H_{K+1}= \lim_{\lambda \rightarrow 1^-} \left( \frac{\partial}{\partial \lambda} \right)^{K+1}\ 
_2\zeta_\lambda(\ _2u_\lambda-1) = -K \left( \frac{K}{d-K} \right)^K \frac{d!}{(d-K)!} <0.
$$
As $K$ is odd, $H_{K+1} (\lambda - 1)^{K+1}/(K+1)!$ is negative, which completes the proof.
\end{proof}

\subsection{Proof of Proposition \ref{prop:LLd1dn_F}} \label{app:LLd1dn_F}
\begin{proof}
The most direct method to show that \eqref{eq:Fbar_LLd1dn} indeed holds is to set $d=\max_{i=1}^n\{d_i\}$ and note that from Theorem 5.2 in \cite{hellemans2019workload} it follows that (for any $w \in [0,\infty)$):
\begin{equation} \label{eq:WL_dependent}
\bar F'(w)=-\lambda \sum_{i=1}^n p_i d_i \P\left\{U \leq w, \mathcal{Q}_i(U) > w \right\},
\end{equation}
where $\mathcal{Q}_i(U)$ represents the workload at an arbitrary queue with workload $U$ after it was one of $d_i$ selected servers for a job arrival. Note that we have:
\begin{align*}
&\P\left\{U \leq w, \mathcal{Q}_i(U) > w \right\}
=\P\left\{0<U \leq w, \mathcal{Q}_i(U) > w \right\} + \P\left\{U =0, \mathcal{Q}_i(U) > w \right\}.
\end{align*}
$\P\left\{0<U \leq w, \mathcal{Q}_i(U) > w \right\}$ is equal to (apply integration by parts):
\begin{align*}
\int_0^w f(u) \bar F(u)^{d_i-1} e^{u-w} \, du 
= \frac{1}{d_i} \bigg( \lambda ^{d_i} e^{-w} - \bar F(w)^{d_i} + \int_0^w \bar F(u)^{d_i} e^{u-w} \, du \bigg)
\end{align*}
with $f(u)$ the density of the workload distribution. For $\P\left\{U =0, \mathcal{Q}_i(U) > w \right\}$ we compute:
\begin{align*}
e^{-w} (1-\bar F(0)) \cdot \sum_{j=0}^{d_i-1} \binom{d_i-1}{j} \frac{(1-\bar F(0))^j \bar F(0)^{d_i-1-j}}{j+1}
= e^{-w} \frac{1-\bar F(0)^{d_i}}{d_i} = e^{-w} \frac{1-\lambda^{d_i}}{d_i}.
\end{align*}
This allows us to conclude, using \eqref{eq:WL_dependent} that:
\begin{equation}
\bar F'(w)
=
-\lambda \sum_{i=1}^n p_i \left( e^{-w} - \bar F(w)^{d_i} + \int_0^w \bar F(u)^{d_i} e^{u-w} \, du \right).\label{eq:WL_dependent_LLd1dn}
\end{equation}
Integrating both sides of \eqref{eq:WL_dependent_LLd1dn} we obtain:
\begin{align*}
\bar F(w) - \bar F(0)
&= -\lambda \sum_{i=1}^n p_i \bigg( 1-e^{-w}-\int_0^w \bar F(u)^{d_i} \, du
+ \int_0^w \int_0^u \bar F(v)^{d_i} e^{v-u} \, dv \, du \bigg)\\
&= - \lambda \sum_{i=1}^n p_i \bigg( 1-e^{-w} - \int_0^w e^{u-w} \bar F(u)^{d_i} \, du \bigg).
\end{align*}
We therefore find that:
$$
\lambda \sum_{i=1}^n p_i \left( e^{-w} + \int_0^w e^{u-w} \bar F(u)^{d_i} \, du \right) = \bar F(w).
$$
Using this to further simplify \eqref{eq:WL_dependent_LLd1dn} allows us to conclude that \eqref{eq:Fbar_LLd1dn} indeed holds.
\end{proof}

\subsection{Proof of Lemma \ref{lem:requirements_SQd1dn}}\label{app:lem:requirements_SQd1dn}
\begin{proof}
Define the function $r(u)=\lambda \sum_{i=1}^n p_i u^{d_i} - u$. We find that $r(1) = \lambda - 1 <0$ and it is obvious that $r(u)$ tends to infinity as $u$ tends to infinity. This shows that there certainly is a $u \in (1,\infty)$ for which $r(u)=0$. Now let:
$$
u_{\lambda} = \min\{u\in (1,\infty) \mid r(u) = 0\}.
$$
We find that for all $u > u_{\lambda}$:
\begin{align*}
r'(u)
&= \lambda \sum_{i=1}^n p_i d_i u^{d_i-1} - 1 > \left(\lambda \sum_{i=1}^n p_i d_i u_{\lambda}^{d_i} - u_{\lambda}\right)/u_{\lambda} 
= \lambda \sum_{i=1}^n p_i (d_i-1) u_{\lambda}^{d_i-1} \geq 0,
\end{align*}
this shows that $r(u) >0$ for all $u > u_{\lambda}$ and hence uniqueness follows. For the other claims we have:
\begin{enumerate}[label=(\alph*),leftmargin=*]
\item This follows from $\lim_{\lambda \rightarrow 1^-} r(1) = 0$.
\item This trivially follows from the fact that $u^{d_i} \leq u$ for all $d_i \geq 1$ and $u \in [0,1]$.
\item For the function $\xi_i(x)$ defined in \ref{lem:requirements_SQd1dn}\ref{enum:d1dn_7} we have
$$
x^2 \cdot \xi_i'(x)= -(u_\lambda)^{d_i} + (u_\lambda - x)^{d_i-1} (u_\lambda + (d_i-1) x ).
$$
Computing the derivative of this, we find:
\begin{align*}
(x^2 \cdot \xi_i'(x))'& = -(d_i-1) d_i (u_\lambda - x)^{d_i-2} x \leq 0,
\end{align*}
for $x \in [u_\lambda-1,u_\lambda]$.
From the fact that $(u_\lambda-1)^2 \xi_i'(u_\lambda-1) = 1+(u_\lambda-1)d_i -(u_\lambda)^{d_i} \leq 0$, we can now incur that $x^2 \cdot \xi_i'(x)$ is negative on $[u_\lambda-1,u_\lambda]$ and therefore
$\xi_i(x)$ is indeed decreasing on $[u_\lambda-1,u_\lambda]$.
This suffices to show  assumption \ref{item:h_decreasing} with
$b=0$ and assumption \ref{item:ODE_only_new_req} with $\bar k =0$ as we have:
\begin{align*}
h_\lambda(x)
&=
\frac{T(u_\lambda) - \lambda \sum_{i=1}^n p_i (u_\lambda - x)^{d_i}}{x}= \lambda \sum_{i=1}^n p_i \xi_i(x). 
\end{align*}
\item First one may compute the limit:
$$
\lim_{\lambda \rightarrow 1^-} h_\lambda(\varepsilon)= \frac{1-\sum_{i=1}^n p_i (1-\varepsilon)^{d_i} }{\varepsilon},
$$
taking the limit of $\varepsilon \rightarrow 0^+$ we obtain:
$$
\lim_{\varepsilon \rightarrow 0^+} \lim_{\lambda\rightarrow 1^-} h_\lambda(\varepsilon) =  \sum_{i=1}^n p_i \lim_{\varepsilon \rightarrow 0^+} \frac{(1-(1-\varepsilon)^{d_i})}{\varepsilon}=\sum_{i=1}^n p_i d_i.
$$
\end{enumerate}
\end{proof}

\subsection{Proof of Proposition \ref{prop:WL_LLd}} \label{app:prop:WL_LLd}
\begin{proof}
Recall from Corollary \ref{cor:gen_result} that the average queue length equals the average workload. The remaining proof goes along the same lines as the proof of Theorem 5.2 in \cite{hellemans2018power} and relies
on the Hypergeometric function ${}_2F_1(a,b,c;z)$ for which the following two properties hold:
\begin{align}
{}_{2}F_1(a,b,c;z)
&=
(1-z)^{-a} \cdot {}_{2}F_1\left(a,c-b,c;\frac{z}{z-1}\right) \label{eq:2F1}\\
{}_2F_1(a,b,c;z)
&= \sum_{n=0}^\infty \frac{(a)_n (b)_n}{(c)_n} \frac{z^n}{n!}  \qquad \mbox{ if } |z|<1.\label{eq:2F1_sum}
\end{align}
Here $(\cdot )_n$ is the Pochhammer symbol (or falling factorial) we have $(q)_n = \prod_{k=0}^{n-1} (q+k)$. We apply \eqref{eq:2F1} to ensure that $z\in (0,1)$ which in turn allows us to apply the sum formula \eqref{eq:2F1_sum}.

The mean workload is given by $\int_{0}^\infty  \bar F(w)\, dw$. Using $y=e^{-w}$ we find that it equals:
\begin{equation}
-\lambda \int_0^1 \left(\frac{b}{p\lambda^d + (b-p\lambda^d)y^{b(d-1)}}\right)^{\frac{1}{d-1}} \frac{1}{y}\, dy, \label{eq:meanWL1}
\end{equation}
with $b=1-(1-p)\lambda$.
By definition of the Hypergeometric function \eqref{eq:meanWL1} is equal to 
$$\frac{\lambda}{b} \left( 1 + \frac{p \lambda^d}{b-p\lambda^d} \right)^{\frac{1}{d-1}} {}_{2}F_1\left( \frac{1}{d-1}, \frac{1}{d-1}, 1 + \frac{1}{d-1}; \frac{-p \lambda^d}{b-p\lambda^d}\right).$$
Equality \eqref{eq:2F1} allows us to rewrite the mean workload as 
$$\frac{\lambda}{b}{}_{2}F_1\left( \frac{1}{d-1}, 1, 1+\frac{1}{d-1}; \frac{p \lambda^d}{b} \right).$$
As $p \lambda^d/b \in (0,1)$, \eqref{eq:2F1_sum} implies that the mean workload is given by:
$$
\frac{\lambda}{b} \sum_{n=0}^\infty \frac{\left(\frac{1}{d-1}\right)_n (1)_n}{\left(1+\frac{1}{d-1}\right)_n} \left( \frac{p \lambda^d}{b} \right)^n.
$$
Using this and the fact that $(1)_n=n!$, we obtain the result.
\end{proof}

\end{document}